\documentclass[10pt]{article}

\input diagram.tex
\usepackage{amsmath} 
\usepackage[mathscr]{eucal}
\usepackage{amssymb} 
\usepackage{theorem}

\theoremstyle{change}  
\newtheorem{theorem}{Theorem}[section] 
\newtheorem{lemma}[theorem]{Lemma}  
\newtheorem{proposition}[theorem]{Proposition}

\theorembodyfont{\rmfamily}  
\newtheorem{Remark}[theorem]{Remark}

\newtheorem{definition}[theorem]{Definition}
\newtheorem{notation}[theorem]{Notation}
\newtheorem{nothing}[theorem]{} 

\newenvironment{proof}{\noindent{\bf Proof}\ }{\qed\bigskip}

\renewcommand{\le}{\leqslant} 
\renewcommand{\ge}{\geqslant}

\newcommand{\alphatilde}{\tilde{\alpha}}
\newcommand{\Aut}{\mathrm{Aut}}

\newcommand{\betatilde}{\tilde{\beta}}

\newcommand{\calB}{\mathcal{B}}
\newcommand{\calC}{\mathcal{C}}
        
\newcommand{\calD}{\mathcal{D}}               

\newcommand{\calF}{\mathcal{F}}

\newcommand{\calM}{\mathcal{M}}

\newcommand{\calS}{\mathcal{S}}

\newcommand{\calT}{\mathcal{T}}
\newcommand{\calX}{\mathcal{X}}
\newcommand{\calXhat}{\hat{\calX}}
\newcommand{\calZ}{\mathcal{Z}}

\newcommand{\catfont}{\mathsf}

\newcommand{\cdotG}{\mathop{\cdot}\limits_{G}}
\newcommand{\cdotH}{\mathop{\cdot}\limits_{H}}

\newcommand{\cdotK}{\mathop{\cdot}\limits_{K}}
\newcommand{\cdott}[1]{\mathop{\cdot}\limits_{#1}}

\newcommand{\defl}{\mathrm{def}}

\newcommand{\fbar}{\bar{f}}

\newcommand{\Hom}{\mathrm{Hom}}

\newcommand{\id}{\mathrm{id}}
\newcommand{\im}{\mathrm{im}}
\newcommand{\Ind}{\mathrm{Ind}}
\newcommand{\ind}{\mathrm{ind}}
\newcommand{\infl}{\mathrm{inf}}

\newcommand{\isom}{\mathrm{iso}}

\newcommand{\KK}{\mathbb{K}}

\newcommand{\lambdabar}{\bar{\lambda}}

\newcommand{\lexp}[2]{\setbox0=\hbox{$#2$} \setbox1=\vbox to
                 \ht0{}\,\box1^{#1}\!#2}

\newcommand{\lMod}[1]{\llap{\phantom{|}}_{#1}\catfont{Mod}}

\newcommand{\myiso}{\buildrel\sim\over\to}

\newcommand{\Out}{\mathrm{Out}}

\newcommand{\Phitilde}{\tilde{\Phi}}
\newcommand{\phibar}{\bar{\phi}}

\newcommand{\phitilde}{\tilde{\phi}}

\newcommand{\psitilde}{\tilde{\psi}}
\newcommand{\qed}{\nobreak\hfill
                  \vbox{\hrule\hbox{\vrule\hbox to 5pt
                  {\vbox to 8pt{\vfil}\hfil}\vrule}\hrule}}
\newcommand{\QQ}{\mathbb{Q}}

\newcommand{\res}{\mathrm{res}}

\newcommand{\stab}{\mathrm{stab}}

\newcommand{\tor}{\mathrm{tor}}

\newcommand{\Ubar}{\bar{U}}

\newcommand{\tw}{\mathrm{tw}}

\newcommand{\zetabar}{\bar{\zeta}}

\newcommand{\ZZ}{\mathbb{Z}}

\title{The $A$-fibered Burnside Ring as $A$-Fibered Biset Functor in Characteristic Zero\footnote{MR Subject Classification 19A22, 20C15; Keywords: Burnside ring, monomial Burnside ring, biset functors, fibered biset functors}}
\author{\small Robert Boltje\thanks{Corresponding author} \ and Deniz Y\i lmaz\\
        \small Department of Mathematics\\
        \small University of California\\
        \small Santa Cruz, CA 95064\\
        \small U.S.A.\\
        \small boltje@ucsc.edu and deyilmaz@ucsc.edu}
\date{September 7, 2019\\ revised August 12, 2020}

\begin{document}

\sloppy

\maketitle


\begin{abstract}
Let $A$ be an abelian group such that $G^*:=\Hom(G,A)$ is finite for all finite groups $G$, and let $\KK$ be a field of characteristic zero which is a splitting field for $G^*$ for all finite groups $G$. In this paper we prove fundamental properties of the $A$-fibered Burnside ring functor $B_\KK^A$ as an $A$-fibered biset functor over $\KK$. This includes a description of the composition factors of $B^A_\KK$ and the lattice of subfunctors of $B_\KK^A$ in terms of what we call $B^A$-pairs and a poset structure on their isomorphism classes. Unfortunately, we are not able to classify $B^A$-pairs. The results of the paper extend results of Co\c{s}kun and Y\i lmaz for the $A$-fibered Burnside ring functor restricted to $p$-groups and results of Bouc in the case that $A$ is trivial, i.e., the case of the Burnside ring functor as a biset functor over fields of characteristic zero. In the latter case, $B^A$-pairs become Bouc's $B$-groups which are also not known in general.
\end{abstract}


\section{Introduction}\label{sec intro}

Let $A$ be a finite group and let $k$ be a commutative ring. An $A$-fibered biset functor $F$ over $k$ is, informally speaking, a functor that assigns to each finite group $G$ a $k$-module $F(G)$ together with maps $\res^G_H\colon F(G)\to F(H)$ and $\ind^G_H\colon F(H)\to F(G)$, whenever $H\le G$, called restriction and induction, maps $\infl_{G/N}^G\colon F(G/N)\to F(G)$ and $\defl_{G/N}^G\colon F(G)\to F(G/N)$, whenever $N$ is a normal subgroup of $G$, called inflation and deflation, and maps $\isom_f\colon F(G)\to F(H)$, whenever $f\colon G\to H$ is an isomorphism. Moreover, the abelian group $G^*:=\Hom(G,A)$ acts $k$-linearly on $F(G)$ for every finite group $G$. These operations satisfy natural relations. Standard examples are various representation rings of $KG$-modules, for a field $K$ and $A=K^\times$. In this case, $G^*$ is the group of one-dimensional $KG$-modules acting by multiplication on these representation rings. In \cite{BoltjeCoskun2018} the simple $A$-fibered biset functors were parametrized. If $A$ is the trivial group then one obtains the well-established theory of biset functors, see~\cite{Bouc2010} as special case. $A$-fibered biset functors over $k$ can also be interpreted as the modules over the Green biset functor $B^A_k$, where $B^A_k(G)$ is the $A$-fibered Burnside ring of $G$ over $k$ (also called the $K$-monomial Burnside ring of $G$ over $k$, when $A=K^\times$ for a field $K$), as introduced by Dress in \cite{Dress1971}. Another natural example of $A$-fibered biset functors (without deflation) is the unit group functor $G\mapsto B^A(G)^\times$. This structure was established in a recent paper by Bouc and Mutlu, see~\cite{BoucMutlu2019} and generalizes the biset functor structure on the unit group $B(G)^\times$ of the Burnside ring.

\smallskip
Representation rings carry more structure when viewed as $A$-fibered biset functors compared to the biset functor structure. One of the goals is to understand their composition factors as such functors. By various induction theorems (e.g.~by Brauer, Artin, Conlon) some of these representation rings can be viewed as quotient functors of the functor $B_k^A$ for appropriate $A$ and $k$. Thus, it is natural to first investigate the lattice of subfunctors of $B_k^A$ and its composition factors. This is the objective of this paper, where $k$ is a field of characteristic zero containing sufficiently many roots of unity. Co\c{s}kun and Y\i lmaz achieved this already in \cite{CoskunYilmaz2019} for the same functor category restricted to finite $p$-groups for fixed $p$, and for $A$ being a cyclic $p$-group. The choice of $A$ being cyclic allows them to use results on primitive idempotents and species of $B_k^A(G)$ from \cite{Barker2004}, which were based on embedding $A$ into $k^\times$, for $k$ a field that is large enough. Our choice of $A$ is more general (using only that $G^*$ is finite for all finite groups $G$), thereby requiring a more complicated parametrizing set for the species and primitive idempotents of $B_k^A(G)$. This way the roles of $A$ and $k$ are kept as separate as possible while still implying that $B_k^A(G)$ is split semisimple.

\smallskip
The paper is arranged as follows. In Section~\ref{sec A-fibered bisets} we recall basic facts about the $A$-fibered Burnside ring $B^A_k(G)$, about fibered biset functors, and the definition of $B^A_k$ as fibered biset functor.
In Section~\ref{sec primitive idempotents} we parametrize the set of primitive idempotents of $B_\KK^A(G)$ over a field $\KK$ of characteristic $0$ which contains enough roots of unity in relation to the finite element orders of $A$. We also derive an explicit formula for these idempotents, using results from \cite{BoltjeRaggiValero2019} on the $-_+$-construction. We take advantage of the fact that the Green biset functor $B^A_\KK$ arises as the $-_+$-construction of the Green biset functor $G\mapsto \KK G^*$. Interestingly, the idempotent formula we derive in Theorem~\ref{thm idempotent formula} is different from the one given by Barker in \cite{Barker2004} when specializing to the more restrictive cases of $A$ considered there. It is used as a crucial tool in the following sections. In Section~\ref{sec elementary operations on primitive idempotents} we provide formulas for the action of inductions, restrictions, inflations, deflations, isomorphisms and twists by $\phi\in G^*$ on these idempotents. Crucial among those is the action of $\defl^G_{G/N}$, which maps a primitive idempotent of $B^A_\KK(G)$ to a scalar multiple of a primitive idempotent of $B^A_{\KK}(G/N)$. After establishing three technical lemmas in Section~\ref{sec two lemmas}, this mysterious scalar is studied in more depth in Section~\ref{sec m}. The vanishing of this scalar is a condition that leads to the notion of a $B^A$-pair $(G,\Phi)$, where $G$ is a finite group and $\Phi\in\Hom(G^*,\KK^\times)$, in Section~\ref{sec 
BA-pairs and the subfunctors E}. There, we also study particular subfunctors $E_{(G,\Phi)}$ of $B_\KK^A$. In Section~\ref{sec subfunctors of BA}, we show that every subfunctor of $B_\KK^A$ is a sum of some of the functors $E_{(G,\Phi)}$ and that the subfunctors of $B_\KK^A$ are in bijection with the set of subsets of isomorphism classes of $B^A$-pairs that are closed from above with respect to a natural partial order $\preccurlyeq$, see Theorem~\ref{thm lattice iso}. In Section~\ref{sec composition factors of BA} we show that the composition factors of $B^A_\KK$ are parametrized by isomorphism classes of $B^A$-pairs. For a given $B^A$-pair, we also determine the parameter (a quadruple) of the corresponding composition factor in terms of the parametrization of all simple $A$-fibered biset functors over $\KK$ given in \cite{BoltjeCoskun2018}. Finally, in Section~\ref{sec A in K}, we consider the special case that $A$ is a subgroup of $\KK^\times$. In this case, a natural isomorphism $G/O(G)\myiso \Hom(G^*,\KK^\times)$ for a normal subgroup $O(G)$ of $G$ depending on $A$, allows to reinterpret the set of $B^A$-pairs and makes our results compatible with the language and setup in \cite{Barker2004} and \cite{CoskunYilmaz2019}.

\smallskip
The approach in this paper follows closely the blueprint in \cite[Section~5]{Bouc2010} for the case $A=\{1\}$. However, the presence of the fiber group $A$ requires additional ideas and techniques to achieve the analogous results. The main technical problem is that a transitive $A$-fibered biset with stabilizer pair $(U,\phi)$, does in general not factor through the group $q(U)\cong p_i(U)/k_i(U)$, $i=1,2$, since $\phi$ is in general non-trivial when restricted to $k_1(U)\times k_2(U)$.

\begin{notation}\label{not intro}
For a finite group $G$ we denote by $\exp(G)$ the exponent of $G$. If $X$ is a left $G$-set, we write $x=_Gy$ if two elements $x$ and $y$ of $X$ belong to the same $G$-orbit. For $x\in X$, we denote by $G_x$ or $\stab_G(x)$ the stabilizer of $x$ in $G$. By $X^G$ we denote the set of $G$-fixed points in $X$ and by $[G\backslash X]$ a set of representatives of the $G$-orbits of $X$. For subgroups $H$ and $K$ of $G$, we denote by $[H\backslash G/K]$ a set of representatives of the $(H,K)$-double cosets of $G$.

For an abelian group $A$, we denote by $\tor(A)$ its subgroup of elements of finite order, and, for a ring $R$, we denote by $R^\times$ its group of units.
\end{notation}


\section{Prerequisites on the $A$-fibered biset functor $B_k^A$}\label{sec A-fibered bisets}

Throughout this paper, we fix a multiplicatively written abelian group $A$. For any finite group $G$ we set $G^*:=\Hom(G,A)$, which we view as abelian group with point-wise multiplication. We will freely use the language of bisets and biset functors as developed in \cite[Chapters 2 and 3]{Bouc2010}. Throughout, $G$, $H$, and $K$ denote finite groups, and $k$ denotes a commutative ring.

\begin{nothing}\label{noth A-fibered Burnside ring}
(a) The {\em $A$-fibered Burnside ring} $B^A(G)$ of $G$ is defined as the Grothendieck group of the category of $A$-fibered left $G$-sets, see \cite[1.1 and 1.7]{BoltjeCoskun2018}. It has a standard $\ZZ$-basis consisting of the $G$-orbits $[U,\phi]_G$ of pairs $(U,\phi)$, where $U\le G$ and $\phi\in U^*$. The set $\calM(G)=\calM^A(G)$ of such pairs has a natural $G$-poset structure, see \cite[1.2]{BoltjeCoskun2018}. The multiplication in $B^A(G)$ is given by
\begin{equation}\label{eqn BAG multiplication}
  [U,\phi]_G\cdot[V,\psi]_G =\sum_{g\in[U\backslash G/V]} 
  \bigl[U\cap\lexp{g}{V}, \phi|_{U\cap\lexp{g}{V}}\cdot(\lexp{g}{\psi})|_{U\cap\lexp{g}{V}}\bigr]_G\,,
\end{equation}
and $[G,1]_G$ is the identity. We set $B_k^A(G):=k\otimes_\ZZ B^A(G)$. If $A$ is the trivial group, we recover the Burnside ring $B(G)$ of $G$.

\smallskip
(b) We further set $B^A(G,H):=B^A(G\times H)$. This group can also be considered as the Grothendieck group of $A$-fibered $(G,H)$-bisets, see \cite[1.1 and 1.7]{BoltjeCoskun2018}. 
The standard basis elements of $B^A(G,H)$ will be denoted by $\left[\frac{G\times H}{U,\phi}\right]$, for $(U,\phi)\in\calM(G\times H)$. The tensor product of $A$-fibered bisets induces a bilinear map
\begin{equation}\label{eqn tensor product}
   -\cdotH -\colon B^A(G,H)\times B^A(H,K)\to B^A(G,K)\,,
\end{equation}
which behaves associatively. Its evaluation on standard basis elements is given by the formula (see \cite[Corollary 2.5]{BoltjeCoskun2018})
\begin{equation}\label{eqn tensor product formula}
   \left[\frac{G\times H}{U,\phi}\right] \cdotH \left[\frac{H\times K}{V,\psi}\right] 
   = \sum_{\substack{t\in [p_2(U)\backslash H/p_1(V)]\\ \phi_2|_{H_t} = \lexp{t}{\psi_1}|_{H_t}}}
   \left[\frac{G\times K}{U*\lexp{(t,1)}{V}, \phi*\lexp{(t,1)}{\psi}}\right]\,,
\end{equation}
where $H_t:= k_2(U)\cap\lexp{t}{k_1(V)}$. See \cite[1.2]{BoltjeCoskun2018} for the definition of $p_i(U)$ and $k_i(U)$, for $i=1,2$,
and \cite[2.3]{BoltjeCoskun2018} for the definition of $U*\lexp{(t,1)}{V}$ and $\phi*\lexp{(t,1)}{\psi}$. Recall from the latter that $\phi_i\in k_i(U)^*$, $i=1,2$, are defined as the unique elements satisfying $\phi|_{k_1(U)\times k_2(U)}=\phi_1\times \phi_2^{-1}$, with $\phi_2^{-1}$ denoting the inverse of $\phi_2$ in the group $k_2(U)^*$.

\smallskip
(c) Recall from \cite[3.1]{CoskunYilmaz2019} that there exists a ring homomorphism $\Delta\colon B^A(G)\to B^A(G,G)$, $[U,\phi]_G\mapsto \left[\frac{G\times G}{\Delta(U),\Delta(\phi)}\right]$, where $\Delta(U):=\{(u,u)\mid u\in U\}$ and $\Delta(\phi)(u,u):=\phi(u)$. Here, $B^A(G,G)$ is considered as ring via the multiplication $\cdot_G$ from (\ref{eqn tensor product}).
\end{nothing}

\begin{nothing}\label{noth A-fibered biset functors}
(a) The bilinear maps in (\ref{eqn tensor product}) are used as the definition of composition in the $k$-linear category $\calC_k^A$, whose objects are the finite groups and whose morphism set from $H$ to $G$ is $B_k^A(G,H)$. An {\em $A$-fibered biset functor over $k$} is a $k$-linear functor $F\colon \calC_k^A\to\lMod{k}$.  Together with natural transformations, these functors form an abelian category $\calF_k^A$. If $A$ is the trivial group, one recovers the biset category $\calC_k$ and the category of biset functors $\calF_k$ over $k$, see \cite[Sections~3.1 and 3.2]{Bouc2010}. 

\smallskip
(b) The association $G\mapsto B_k^A(G)$, together with the bilinear maps in (\ref{eqn tensor product}) applied to $K=1$ (and using the canonical identifications of $H\cong H\times\{1\}$ and $G\cong G\times \{1\}$) defines the $A$-fibered biset functor $B^A_k\in\calF_k^A$, which is the main object of study in this paper.

\smallskip
(c) If $f\colon A'\to A$ is a homomorphism between abelian groups, one obtains induced maps $\Hom(G,A')\to \Hom(G,A)$, $\calM^{A'}(G)\to\calM^A(G)$, ring homomorphisms $B^{A'}(G)\to B^A(G)$, and $k$-linear functors $\calC_k^{A'}\to\calC_k^A$ and $\calF_k^{A}\to\calF_k^{A'}$. In particular, when $A'$ is the trivial group, we obtain embeddings $B(G,H)\subseteq B^A(G,H)$, $[(G\times H)/U]\mapsto \left[\frac{G\times H}{U,1}\right]$, and $\calC_k\subseteq\calC_k^A$. This way, we can view the elementary bisets $\res^G_H$, $\ind_H^G$, $\infl_{G/N}^G$, $\defl^G_{G/N}$, $\isom_f$, for $H\le G$, $N\trianglelefteq G$, $f\colon G\myiso G'$, as elements in $B^A(G,H)$, $B^A(H,G)$, $B^A(G,G/N)$, $B^A(G/N,G)$, $B^A(G',G)$, respectively. One can verify quickly with (\ref{eqn tensor product formula}) for $K=\{1\}$, that their operations under the $A$-fibered biset functor structure of $B^A_k$ in (b) on standard basis elements are given by
\begin{equation*}
  \res^G_H([U,\phi]_G) = \sum_{g\in [H\backslash G/U]} [H\cap\lexp{g}{U}, (\lexp{g}{\phi})|_{H\cap\lexp{g}{U}}]_H\,,\quad
  \ind_H^G([V,\psi]_H) = [V,\psi]_G\,,
\end{equation*}
\begin{equation*}
  \infl_{G/N}^G([U/N,\phi]_{G/N}) = [U,\phi\circ\nu]_G\,,
\end{equation*}
where $\nu\colon U\to U/N$ is the natural epimorphism,
\begin{equation*}
  \defl^G_{G/N}([U,\phi]_G)=
  \begin{cases} [UN/N,\phitilde]_{G/N}, &\text{if $U\cap N\le \ker(\phi)$,}\\ 0 & \text{otherwise,} \end{cases}
\end{equation*}
where $\phitilde(uN):=\phi(u)$ for $u\in U$, and
\begin{equation*}
   \isom_f([U,\phi]_G) = [f(U),\phi\circ f^{-1}|_{f(U)}]_{G'}\,.
\end{equation*}
In particular, for $g\in G$ and $H\le G$, and $(U,\phi)\in\calM(H)$, we have
\begin{equation*}
   \lexp{g}{[U,\phi]_H} := \isom_{c_g} ([U,\phi]_H) = [\lexp{g}{U},\lexp{g}{\phi}]_{\lexp{g}{H}}\,,
\end{equation*}
where $c_g\colon H\myiso gHg^{-1}$ is the conjugation map.
For $\lambda\in G^*$ and $(U,\phi)\in \calM(G)$, one additionally has 
\begin{equation*}
  \tw_\lambda([U,\phi]_G)=[U,\lambda|_U\cdot\phi]_G\,,
\end{equation*}
the {\em twist} by $\lambda$, coming from the application of $\tw_\lambda:=\Delta([G,\lambda]_G)=\left[\frac{G\times G}{\Delta(G),\Delta(\lambda)}\right]\in B^A(G,G)$.
It is easily verified that $\res^G_H$, $\inf_{G/N}^G$ and $\isom_f$ are ring homomorphisms, and that $B^A_k(G)$ is a $kG^*$-algebra via $\tw$.
\end{nothing}

Using the above notation, one obtains a canonical decomposition of a standard basis element of $B^A(G,H)$ into elementary bisets and a standard basis element for smaller groups.

\begin{theorem}\label{thm can decomp}(\cite[Proposition~2.8]{BoltjeCoskun2018})
Let $(U,\phi)\in\calM(G\times H)$ and set $P:=p_1(U)$, $Q:=p_2(U)$, $K:=\ker(\phi_1)$, and $L:=\ker(\phi_2)$. Then $K\trianglelefteq P$, $L\trianglelefteq Q$, $K\times L\trianglelefteq U$, and 
\begin{equation*}
  \left[\frac{G\times H}{U,\phi}\right] = \ind_P^G \cdott{P} \infl_{P/K}^P \cdott{P/K}
  \left[\frac{P/K\times Q/L}{U/(K\times L),\phibar}\right] \cdott{Q/L}
  \defl^Q_{Q/L} \cdott{Q} \res^H_L\,,
\end{equation*}
where $\phibar\in (U/(K\times L))^*$ is induced by $\phi$ and $U/(K\times L)$ is viewed as subgroup of $P/K\times Q/L$ via the canonical isomorphism $(P\times Q)/(K\times L)\cong P/K\times Q/L$.
\end{theorem}

\begin{Remark}\label{rem + construction}
Let $\calD_k$ be the subcategory of $\calC_k$ with the same objects as $\calC_k$, but with morphism sets generated by all elementary bisets, excluding inductions. In other words, $\Hom_{\calD_k}(H,G)\subseteq B_k(G,H)$ is the free $k$-module generated by all standard basis elements $[(G\times H)/U]$ with $p_1(U)=G$. Mapping $G$ to the group algebra $kG^*$ defines a Green biset functor $F$ on $\calD_k$ over $k$ in the sense of Romero's reformulation \cite[Definici\'on 3.2.7, Lema 4.2.3]{Romero2011} of Bouc's original definition \cite[Definition~8.5.1]{Bouc2010}, with restriction, inflation and isomorphisms defined as usual, viewing $\Hom(-,A)$ as contravariant functor, and deflation defined by
\begin{equation*}
   \defl^G_{G/N}(\phi):=\begin{cases} \phibar, \text{if $\phi|_N=1$,}\\ 0, \text{otherwise,}\end {cases}
\end{equation*}
whenever $N$ is a normal subgroup of $G$ and $\phi\in G^*$. Here, $\phibar\in (G/N)^*$ is induced by $\phi$. In fact, it is straightforward to check that all the relations in \cite[1.1.3]{Bouc2010} that do not involve inductions are satisfied. Thus, we are in the situation of \cite[Theorem~7.3(a),(b)]{BoltjeRaggiValero2019} and obtain via the $-_+$-construction a Green biset functor $F_+$ on $(\calD_k)_+=\calC_k$. The Green biset functor $F_+$ is isomorphic to the Green biset functor $B_k^A$ on $\calC_k$. This follows from \cite[Theorem~4.7(c)]{BoltjeRaggiValero2019} and by comparing the explicit formulas for the elementary biset operations in \ref{noth A-fibered biset functors}(c) with the explicit formulas in \cite[Remark~4.8]{BoltjeRaggiValero2019}. We will use this point of view repeatedly in Sections~\ref{sec primitive idempotents} and \ref{sec elementary operations on primitive idempotents}.
\end{Remark}


\section{Primitive idempotents of $B_\KK^A(G)$}\label{sec primitive idempotents}

Throughout this section we assume that $G$ is a finite group such that $H^*=\Hom(H,A)$ is a finite abelian group for every $H\le G$. This is equivalent to $\tor_{\exp(G)}(A)$ being finite. Moreover, we assume that $\KK$ is a splitting field of characteristic zero for all $H^*$, $H\le G$. Note that this holds if and only if $\KK$ has a root of unity of order $\exp(\tor_{\exp(G)}(A))$. Also note that in this case $S^*$ is finite and $\KK$ is a splitting field for $S^*$, for each subquotient $S$ of $G$.

\smallskip
We define $\calX(G)$ as the set of all pairs $(H,\Phi)$ with $H\le G$ and $\Phi\in \Hom(H^*,\KK^\times)$ and note that $G$ acts on $\calX(G)$ by conjugation: $\lexp{g}{(H,\Phi)}:=(\lexp{g}{H},\lexp{g}{\Phi})$, with $\lexp{g}{\Phi}(\phi):=\Phi(\lexp{g^{-1}}{\phi})$, for $g\in G$, $(H,\Phi)\in\calX(G)$, and $\phi\in H^*$. The assumptions on $\KK$ imply that, for any $H\le G$,
\begin{equation}\label{eqn isom 1}
   \KK H^*\to \prod_{\Phi\in\Hom(H^*,\KK^\times)} \KK\,, \quad a\mapsto (s_\Phi^H(a))_{\Phi}\,,
\end{equation}
is an isomorphism of $\KK$-algebras. Here, we $\KK$-linearly extended $\Phi$ to a $\KK$-algebra homomorphism 
\begin{equation*}
  s_\Phi^H\colon \KK H^*\to \KK\,.
\end{equation*} 
The first orthogonality relation implies that, for $\Psi\in\Hom(H^*,\KK^\times)$, the element
\begin{equation}\label{eqn idempotent 1}
   e_\Psi^H:=\frac{1}{|H^*|} \sum_{\phi\in H^*} \Psi(\phi^{-1})\phi\in \KK H^*
\end{equation}
is the primitive idempotent of $\KK H^*$ which is mapped under the isomorphism in (\ref{eqn isom 1}) to the primitive idempotent $\epsilon_\Psi^H\in\prod_{\Phi}\KK$ whose $\Phi$-component is $\delta_{\Phi,\Psi}$.

For any $H\le G$ we consider the map
\begin{equation*}
  \pi_H\colon B_\KK^A(H)\to \KK H^*\,,\quad 
  [U,\phi]_H\mapsto\begin{cases} \phi & \text{if $U=H$,}\\ 0 & \text{otherwise.} \end{cases}
\end{equation*}
It is easily seen by the multiplication formula in (\ref{eqn BAG multiplication}) that $\pi_H$ is a $\KK$-algebra homomorphism and we obtain for every $(H,\Phi)\in\calX(G)$, a $\KK$-algebra homomorphism
\begin{equation*}
  s_{(H,\Phi)}^G:=s_\Phi^H\circ \pi_H \circ \res^G_H\colon B_\KK^A(G)\to B_\KK^A(H)\to \KK H^* \to \KK\,.
\end{equation*}

\begin{theorem}\label{thm mark isomorphism}
The map
\begin{equation}\label{eqn mark isomorphism}
  B_\KK^A(G)\to \left(\prod_{(H,\Phi)\in\calX(G)} \KK\right)^G\,,\quad 
  x\mapsto \bigl(s_{(H,\Phi)}^G(x)\bigr)_{(H,\Phi)}\,,
\end{equation}
is a $\KK$-algebra isomorphism. Here, $G$ acts on $\prod_{(H,\Phi)\in\calX(G)}\KK$ by permuting the coordinates according to the $G$-action on $\calX(G)$. In particular, every $\KK$-algebra homomorphism $B_\KK^A(G)\to \KK$ is of the form $s_{(H,\Phi)}^G$ for some $(H,\Phi)\in\calX(G)$. For $(H,\Phi),(K,\Psi)\in\calX(G)$ one has $s_{(H,\Phi)}^G=s_{(K,\Psi)}^G$ if and only if $(H,\Phi)=_G (K,\Psi)$.
\end{theorem}

\begin{proof}
By Theorem~\cite[Theorems~6.1 and 7.3(c)]{BoltjeRaggiValero2019} and using Remark~\ref{rem + construction}, the {\em mark morphism}
\begin{equation}\label{eqn isom 3}
   m_G\colon B_\KK^A(G)\to \left(\prod_{H\le G} \KK H^*\right)^G\,, \quad x\mapsto \bigl(\pi_H(\res^G_H(x))\bigr)_{H\le G}\,,
\end{equation}
is a homomorphism of $\KK$-algebras and by Theorem~\cite[Corollary~6.4]{BoltjeRaggiValero2019} it is an isomorphism, since $|G|$ is invertible in $\KK$. Here, $G$ acts on $\prod_{H\le G} \KK H^*$ by 
$\lexp{g}{((a_H)_{H\le G})}:=(\lexp{g}{a_{g^{-1}Hg}})_{H\le G}$. Using the $\KK$-algebra isomorphisms from (\ref{eqn isom 1}), we obtain a $G$-equivariant $\KK$-algebra isomorphism
\begin{equation*}
   \prod_{H\le G} \KK H^* \to \prod_{(H,\Phi)\in\calX(G)} \KK\,.
\end{equation*}
Applying the functor of $G$-fixed points to this isomorphism and composing it with the isomorphism in (\ref{eqn isom 3}), we obtain the isomorphism in (\ref{eqn mark isomorphism}). The remaining assertions follow immediately.
\end{proof}

Clearly, for each $(H,\Phi)\in\calX(G)$, we obtain a primitive idempotent $\epsilon_{(H,\Phi)}^G$ of the right hand side of the isomorphism (\ref{eqn mark isomorphism}). More precisely, $\epsilon_{(H,\Phi)}^G$ has entries equal to $1$ at indices labelled by the $G$-conjugates of $(H,\Phi)$ and entries equal to $0$ everywhere else. We denote the idempotent of $B_\KK^A(G)$ corresponding to $\epsilon_{(H,\Phi)}^G$ by $e_{(H,\Phi)}^G\in B_\KK^A(G)$. If $(H,\Phi)$ runs through a set of representatives of the $G$-orbits of $\calX(G)$ then $e_{(H,\Phi)}^G$ runs through the set of primitive idempotents of $B_\KK^A(G)$, without repetition. Thus, we have
\begin{equation}\label{eqn s and e}
   s_{(H,\Phi)}^G(e_{(K,\Psi)}^G) = \begin{cases} 1, & \text{if $(H,\Phi)=_G(K,\Psi)$,}\\ 0, & \text{otherwise,} \end{cases}
   \quad\text{and}\quad
   x\cdot e_{(H,\Phi)}^G = s_{(H,\Phi)}^G(x) e_{(H,\Phi)}^G\,,
\end{equation}
for any $(H,\Phi),(K,\Psi)\in\calX(G)$ and any $x\in B_\KK^A(G)$.

\medskip
The following theorem gives an explicit formula for $e_{(H,\Phi)}^G$. A different formula for particular choices of $A$ was given by Barker in \cite[Theorem~5.2]{Barker2004}.
For any $H\le G$ and $a=\sum_{\phi\in H^*} a_\phi \phi\in \KK H^*$ we will use the notation $[H,a]_G:=\sum_{\phi\in H^*}a_\phi [H,\phi]_G\in B_\KK^A(G)$. Moreover, $N_G(H,\Phi)$ denotes the stabilizer of $(H,\Phi)$ under $G$-conjugation.

\begin{theorem}\label{thm idempotent formula}
For $(H,\Phi)\in\calX(G)$ one has
\begin{align}\label{eqn idempotent formula 1}
  e_{(H,\Phi)}^G & = \frac{1}{|N_G(H,\Phi)|} \sum_{K\le H} |K|\mu(K,H) [K,\res^H_K(e_\Phi^H)]_G\\
  \label{eqn idempotent formula 2}
  & = \frac{1}{|N_G(H,\Phi)|} \sum_{\substack{K\le H\\ \Phi|_{K^\perp}=1}} |K|\mu(K,H) [K,\res^H_K(e_\Phi^H)]_G\\
  \label{eqn idempotent formula 3}
  & =\frac{1}{|N_G(H,\Phi)|\cdot|H^*|}
       \sum_{\substack{K\le H\\ \Phi|_{K^\perp}=1}} \sum_{\phi\in H^*} |K|\mu(K,H)\Phi(\phi^{-1})[K,\phi|_K]_G\in B_\KK^A(G)\,,
\end{align}
where $K^\perp:=\{\phi\in H^*\mid \phi|_K=1\}\le H^*$ and $\mu$ is the M\"obius function on the poset of all subgroups of $G$.
\end{theorem}

\begin{proof}
We use the inversion formula of the $\KK$-algebra isomorphism (\ref{eqn isom 3}) from \cite[Proposition~6.3]{BoltjeRaggiValero2019} and obtain
\begin{equation}\label{eqn idempotent formula 4}
   e_{(H,\Phi)}^G=\frac{1}{|G|}\sum_{L\le K\le G} |L|\mu(L,K) [L,\res^K_L(a_K)]_G\,,
\end{equation}
with $a_K\in \KK K^*$, $K\le G$, given by $a_H=\sum_{x\in [N_G(H)/N_G(H,\Phi)]} \lexp{x}{e^H_{\Phi}}$, $a_{\lexp{g}{H}}=\lexp{g}{a_H}$, for any $g\in G$, and $a_K=0$ for all $K$ not $G$-conjugate to $H$. Thus, in the above sum, for $K$ it suffices to consider only subgroups that are $G$-conjugate to $H$. We obtain
\begin{equation}\label{idempotent formula 5}
   e_{(H,\Phi)}^G=\frac{1}{|G|}
   \sum_{x\in [G/N_G(H)]}\sum_{L\le \lexp{x}{H}} |L|\mu(L,\lexp{x}{H}) [L,\res^{\lexp{x}{H}}_L(\lexp{x}{a_H})]_G\,.
\end{equation}
Replacing $L$ with $\lexp{x}{K}$, for $K\le H$, we see that the sum over $L$ is independent of $x$ and we obtain
\begin{equation}\label{idempotent formula 6}
   e_{(H,\Phi)}^G=\frac{[G:N_G(H)]}{|G|}
   \sum_{K\le H} |K|\mu(K,H) [K,\res^H_K(a_H)]_G\,.
\end{equation}
Substituting $a_H=\sum_{g\in [N_G(H)/N_G(H,\Phi)]}\lexp{g}{e_\Phi^H}$ and using the same argument as above, we obtain
the formula in (\ref{eqn idempotent formula 1}). In order to prove the formula in (\ref{eqn idempotent formula 2})  it suffices to show that $\res^H_K(e_\Phi^H)=0$ if $\Phi|_{K^\perp}\neq 1$. Substituting the formula (\ref{eqn idempotent 1}) for $e_\Phi^H$, we obtain
\begin{equation}\label{eqn rest formula}
   \res^H_K(e_\Phi^H)=\frac{1}{|H^*|}\sum_{\phi\in H^*} \Phi(\phi^{-1}) \phi|_K\,.
\end{equation}
Note that $K^\perp=\ker(\res^H_K\colon H^*\to K^*)$ and choose for every $\psi\in\im(\res^H_K)\le K^*$ an element $\psitilde\in H^*$ with $\psitilde|_K=\psi$. Then the right hand side in (\ref{eqn rest formula}) is equal to
\begin{equation*}
   \frac{1}{|H^*|}\sum_{\psi\in\im(\res^H_K)}\sum_{\lambda\in K^\perp} \Phi(\psitilde^{-1}\lambda^{-1})\psi = 
   \frac{1}{|H^*|}\sum_{\psi\in\im(\res^H_K)}\Phi(\psitilde^{-1})\bigl(\sum_{\lambda\in K^\perp} \Phi(\lambda^{-1})\bigr) \psi\,,
 \end{equation*}
and it suffices to show that $\sum_{\lambda\in K^\perp} \Phi|_{K^\perp}(\lambda^{-1})=0$. But 
\begin{equation}\label{eqn rest formula 2}
   \sum_{\lambda\in K^\perp} \Phi|_{K^\perp}(\lambda^{-1}) = 
   [K^\perp:K^\perp\cap \ker(\Phi)]\sum_{\lambdabar\in K^\perp/(K^\perp\cap \ker(\Phi))} 
   \overline{\Phi|_{K^\perp}}(\lambdabar^{-1})\,,
\end{equation}
with an injective homomorphism from $\overline{\Phi|_{K^\perp}}\colon K^\perp/(K^\perp\cap \ker(\Phi))\to \KK^\times$. It follows that $K^\perp/(K^\perp\cap\ker(\Phi))$ is cyclic, say of order $n$. Our assumption on $\KK$ implies that $\KK$ has a primitive $n$-th root of unity. Moreover, since $\Phi|_{K^\perp}\neq 1$, we have $n>1$. Thus the sum on the right hand side of (\ref{eqn rest formula 2})  is equal to the sum of all $n$-th roots of unity in $\KK$, which is $0$. This proves Equation~(\ref{eqn idempotent formula 2}). Formula (\ref{eqn idempotent formula 3}) is now immediate after substituting the formula for $e_\Phi^H$.
\end{proof}

\begin{Remark}\label{rem comparing B and BA}
If $A'$ is the trivial subgroup of $A$ we have $B_\KK^{A'}(G) = B_\KK(G)$, the Burnside algebra over $\KK$. Using the functoriality properties in \ref{noth A-fibered biset functors}(c), we obtain a commutative diagram
\begin{diagram}[85]
   B_\KK(G) & \Ear{m_G} & \movevertex(0,-8){(\mathop{\prod}\limits_{H\le G} \KK)^G} &\movearrow(5,0){ \Ear[40]{\id} }& 
        \movevertex(10,-8){(\mathop{\prod}\limits_{H\le G} \KK)^G} &&
   \sar & & \sar & & \movearrow(10,0){\sar} &&
   B_\KK^A(G) & \Ear{m_G} & \movevertex(0,-8){(\mathop{\prod}\limits_{H\le G} \KK H^*)^G} & \ear[20] & 
       \movevertex(10,-8){(\mathop{\prod}\limits_{(H,\Phi)\in\calX(G)} \KK)^G} &&
\end{diagram}
of $\KK$-algebra homomorphisms, where the left horizontal maps are the mark isomorphisms $m_G$ from (\ref{eqn isom 3}) given by $(\pi_H\circ\res^G_H)$, the right top horizontal map is the identity, the right bottom horizontal map is the product of the isomorphisms from (\ref{eqn isom 1}), the middle vertical map is the product of the unique $\KK$-algebra homomorphisms $\KK\to\KK H^*$, and the right vertical map is induced by the $G$-equivariant map $\calX(G)\to \{H\le G\}$, $(H,\Phi)\mapsto H$, between the indexing sets. We denote the primitive idempotents of $B_\KK(G)$ by $e_H^G$, for any $H\le G$. Thus, by the above commutative diagram,
\begin{equation}\label{eqn B and BA idempotents}
   e_G^G=\sum_{\Phi\in\Hom(G^*,\KK^\times)} e_{(G,\Phi)}^G\,.
\end{equation}
\end{Remark}

\begin{lemma}\label{lem pi and e}
For any $x\in B_\KK^A(G)$ one has $e_G^G\cdot x=0$ if and only if $\pi_G(x)=0$.
\end{lemma}

\begin{proof}
Since $m_G\colon B^A_\KK(G)\to \prod_{H\le G}\KK$ is injective and multiplicative, one has $e_G^G\cdot x=0$ if and only if $m_G(e_G^G)\cdot m_G(x)=0$. But $m_G(e_G^G)$ has entry $1$ in the $G$-component and entry $0$ everywhere else. Thus, $e_G^G\cdot x=0$ if and only if the entry of $m_G(x)$ in the $G$-component is equal to $0$. But this entry equals $\pi_G(x)$.
\end{proof}


\section{Elementary operations on primitive idempotents}\label{sec elementary operations on primitive idempotents}

Throughout this section we assume as in Section~\ref{sec primitive idempotents} that $G$ is a finite group such that $S^*=\Hom(S,A)$ is finite for all subquotients $S$ of $G$, and that $\KK$ is a field of characteristic $0$ which is a splitting field of $S^*$ for all subquotients $S$ of $G$. 

In this section we will establish formulas for elementary fibered biset operations on the primitive idempotents of $B^A_\KK(S)$ for subquotients $S$ of $G$. These formulas will be used in later sections.

\begin{proposition}\label{prop s comp res}
Let $H\le G$.

\smallskip
{\rm (a)} For $(L,\Psi)\in\calX(H)$ one has $s_{(L,\Psi)}^H\circ \res^G_H = s_{(L,\Psi)}^G\colon B^A_\KK(G)\to \KK$.

\smallskip
{\rm (b)} For $(K,\Phi)\in\calX(G)$ one has
\begin{equation*}
  \res^G_H(e_{(K,\Phi)}^G) = \sum_{\substack{(L,\Psi)\in[H\backslash \calX(H)] \\ (L,\Psi)=_G (K,\Phi)}} e_{(L,\Psi)}^H\,.
\end{equation*}

\smallskip
{\rm (c)} For $\Phi\in \Hom(G^*,\KK^\times)$ and $H<G$ one has $\res^G_H(e_{(G,\Phi)}^G)=0$.
\end{proposition}

\begin{proof}
(a) We use the point of view from Remark~\ref{rem + construction}. By \cite[Equation~(13) and Theorem~6.1]{BoltjeRaggiValero2019} the left square in
\begin{diagram}[85]
  B_\KK^A(G) & \Ear{m_G} & \movevertex(0,-8){\mathop{\prod}\limits_{K\le G} \KK K^*} & \ear & 
  \movevertex(0,-8){\mathop{\prod}\limits_{(K, \Phi)\in\calX(G)} \KK} &&
  \Sar{\res^G_H} & & \sar & & \sar &&
  B_\KK^A(H) & \Ear{m_H} & \movevertex(0,-8){\mathop{\prod}\limits_{L\le H} \KK L^*} & \ear & 
  \movevertex(0,-8){\mathop{\prod}\limits_{(L,\Psi)\in\calX(H)} \KK} &&
\end{diagram}
is commutative, where the left horizontal maps are the mark homomorphisms from (\ref{eqn isom 3}), the right horizontal maps are given by the isomorphisms in (\ref{eqn isom 1}), and the middle and right  vertical maps are the canonical projections. Since the right hand square commutes as well, following up with the projection onto the $(L,\Psi)$-component of $\prod_{(L,\Psi)\in \calX(H)}\KK$, yields the result.

\smallskip
(b) Since $\res^G_H\colon B_\KK^A(G)\to B_\KK^A(H) $ is a $\KK$-algebra homomorphism, $\res^G_H(e_{(K,\Phi)})$ is the sum of certain primitive idempotents $e_{(L,\Psi)}^H$, $(L,\Psi)\in[H\backslash \calX(H)]$. Moreover, $e_{(L,\Psi)}^H$ occurs in this sum if and only if $s_{(L,\Psi)}^H(\res^G_H(e_{(K,\Phi)}^G))\neq 0$. The result follows now immediately from (a).

\smallskip
(c) This follows immediately from Part~(b).
\end{proof}

\begin{proposition}\label{prop s comp inf}
Let $N\trianglelefteq G$.

\smallskip
{\rm (a)} For $(H,\Phi)\in\calX(G)$ one has $s_{(H,\Phi)}^G\circ \infl_{G/N}^G = s_{(HN/N,\Phi_N)}^{G/N}$, where $\Phi_N:=\Phi\circ\nu^*\circ \alpha^*\in\Hom((HN/N)^*,\KK^\times)$ with $\alpha\colon H/(H\cap N)\myiso HN/N$ denoting the canonical isomorphism and $\nu\colon H\to H/(H\cap N)$ denoting the canonical epimorphism.

\smallskip
{\rm (b)} For $(U/N,\Psi)\in \calX(G/N)$ with $N\le U\le G$, one has
\begin{equation*}
   \infl_{G/N}^G(e_{(U/N,\Psi)}^{G/N}) = \sum_{\substack{(H,\Phi)\in[G\backslash \calX(G)] \\ (HN/N,\Phi_N) =_{G/N} (U/N,\Psi)}}
   e_{(H,\Phi)}^G\,.
\end{equation*}
\end{proposition}

\begin{proof}
(a) We use again the point of view from Remark~\ref{rem + construction}.
By \cite[Equation~(12)]{BoltjeRaggiValero2019} applied to $D:=\{(g,gN)\mid g\in G\}\le G\times G/N$ and \cite[Theorem~6.1]{BoltjeRaggiValero2019} the left square in
\begin{diagram}[85]
  \movevertex(-20,0){B_\KK^A(G/N)} &\movearrow(-17,0){\Ear[33]{m_{G/N}}} & 
      \movevertex(0,-8){\mathop{\prod}\limits_{N\le U\le G} \KK (U/N)^*} & 
      \movearrow(20,0){\ear[20]} & \movevertex(30,-8){\mathop{\prod}\limits_{(U/N, \Psi)\in\calX(G/N)} \KK} &&
  \movearrow(-20,0){\Sar{\infl^G_{G/N}} }& & \sar & & \movearrow(30,0){\sar} &&
  \movevertex(-20,0){B_\KK^A(G)} & \movearrow(-15,0){\Ear[40]{m_G}} & 
      \movevertex(0,-8){\mathop{\prod}\limits_{H\le G} \KK H^*} & 
      \movearrow(15,0){\ear[40]}& \movevertex(30,-8){\mathop{\prod}\limits_{(H,\Phi)\in\calX(G)} \KK} &&
\end{diagram}
is commutative, where the left horizontal maps are the mark homomorphisms from (\ref{eqn isom 3}), the right horizontal maps are given by the isomorphisms in (\ref{eqn isom 1}), the middle vertical homomorphism maps the family $(a_{U/N})_{N\le U\le G}$ to $(\infl_{H/(H\cap N)}^H(\alpha^*(a_{HN/N})))_{H\le G}$ with $\alpha\colon H/(H\cap N)\myiso HN/N$ denoting the canonical isomorphism, and the right vertical homomorphism maps the family $(a_{(U/N,\Psi)})_{(U/N,\Psi)\in\calX(G/N)}$ to the family $(a_{(HN/N,\Phi_N)})_{(H,\Phi)\in\calX(G)}$.
Since the right hand square commutes as well (note that $\inf_{H/(H\cap N)}^H\colon \KK (H/(H\cap N))^*\to \KK H^*$ is the $\KK$-linear extension of $\nu^*$ from (a)), following up with the projection onto the $(H,\Phi)$-component of $\prod_{(H,\Phi)\in \calX(G)}\KK$, yields the result.

\smallskip
(b) Since $\infl_{G/N}^G\colon B_\KK^A(G/N)\to B_\KK^A(G)$ is a $\KK$-algebra homomorphism, $\infl_{G/N}^G(e_{(U/N,\Psi)}^{G/N})$ is the sum of certain primitive idempotents $e_{(H,\Phi)}^G$, $(H,\Phi)\in [G\backslash \calX(G)]$. Moreover, $e_{(H,\Phi)}^G$ occurs in this sum if and only if $s_{(H,\Phi)}^G(\inf_{G/N}^G(e_{(U/N,\Psi)}^{G/N}))\neq 0$. Part~(a) now implies the result. 
\end{proof}

\begin{proposition}\label{prop m number}
Let $N\trianglelefteq G$.

\smallskip
{\rm (a)} For all $(H,\Phi)\in\calX(G)$, there exists $m_{(H,\Phi)}^{G,N}\in \KK$ such that
\begin{equation}\label{eqn m equation}
  \defl^G_{G/N}(e_{(H,\Phi)}^G) = m_{(H,\Phi)}^{G,N}\cdot e_{(HN/N,\Phi_N)}^G
\end{equation}
with $\Phi_N$ defined as in Proposition~\ref{prop s comp inf}(a).

\smallskip
{\rm (b)} For all $\Phi\in G^*$ one has
\begin{equation}\label{eqn m formula}
   m_{(G,\Phi)}^N:= m_{(G,\Phi)}^{G,N} = \frac{|(G/N)^*|}{|G|\cdot|G^*|} 
   \sum_{\substack{ K\le G \\ KN=G \\ \Phi|_{K^\perp}=1}} |K|\cdot|K^\perp|\cdot\mu(K,G)\in \QQ\,.
\end{equation}
\end{proposition}

\begin{proof}
(a)  For any $x\in B_\KK^A(G/N)$ we have
\begin{align*}
   x\cdot \defl^G_{G/N}(e_{(H,\Phi)}^G) & = \defl^G_{G/N}(\infl^G_{G/N}(x)\cdot e_{(H,\Phi)}^G) = 
        \defl^G_{G/N}(s_{(H,\Phi)}^G(\infl_{G/N}^G(x))\cdot e_{(H,\Phi)}^G) \\
   & = \defl^G_{G/N}(s_{(HN/N,\Phi_N)}^{G/N}(x)\cdot e_{(H,\Phi)}^G) = 
        s_{(HN/N,\Phi_N)}^{G/N}(x)\cdot \defl^G_{G/N}(e_{(H,\Phi)}^G)\,.
\end{align*}
In fact, the first equation follows from the Green biset functor axioms (see \cite[Definition~7.2(a)]{BoltjeRaggiValero2019} and \cite[Definici\'on~3.2.7, Lema~4.2.3]{Romero2011}), the second from (\ref{eqn s and e}), and the third from Proposition~\ref{prop s comp inf}(a). Choosing $x=e_{(HN/N,\Phi_N)}^{G/N}$, and reading the above equations backward, we obtain \begin{equation*}
   \defl^G_{G/N}(e_{(H,\Phi)}^G) = \defl^G_{G/N}(e_{(H,\Phi)}^G)\cdot e_{(HN/N,\Phi_N)}^G\,.
\end{equation*} 
Now, (\ref{eqn s and e}) implies the result with $m_{(H,\Phi)}^{G,N}=s_{(HN/N,\Phi_N)}^{G/N}(\defl^G_{G/N}(e_{(H,\Phi)}^G))$.

\smallskip
(b) Substituting the explicit idempotent formula (\ref{eqn idempotent formula 3}) for $e_{(G,\Phi)}^G$ and using the explicit formula for $\defl^G_{G/N}\colon B_\KK^G(G)\to B_\KK^A(G/N)$ from \ref{noth A-fibered biset functors}(c), the left hand side of (\ref{eqn m equation}) is equal to
\begin{equation*}
   \frac{1}{|G|\cdot|G^*|} \sum_{\substack{ K\le G\\ \Phi|_{K^\perp}=1}} \ \sum_{\substack{\phi\in G^*\\ \phi|_{K\cap N}=1}} 
   |K|\,\mu(K,G)\,\Phi(\phi^{-1})\, [KN/N, \widetilde{\phi|_K}]_{G/N}\,,
\end{equation*}
where $\widetilde{\phi|_K}(kN):=\phi(k)$ for $k\in K$.
Moreover, using the explicit formula (\ref{eqn idempotent formula 3}) for $e_{(G/N,\Phi_N)}^{G/N}$, the right hand side of (\ref{eqn m equation}) is equal to
\begin{equation*}
   \frac{m^N_{(G,\Phi)}}{|G/N|\cdot|(G/N)^*|}\sum_{\substack{U/N\le G/N\\ \Phi_N|_{(U/N)^\perp}=1}}\ \sum_{\psi\in(G/N)^*}
   |U/N|\,\mu(U/N,G/N)\,\Phi_N(\psi^{-1})\, [U/N,\psi|_{U/N}]_{G/N}\,.
\end{equation*}
Next we compare the coefficients at the standard basis element $[G/N,1]_{G/N}$ of $B_\KK^A(G/N)$ on both sides. On the left hand side, we only have to sum over those $K\le G$ with $KN=G$ and those $\phi\in G^*$ with $\widetilde{\phi|_K} = 1$. By the definition of $\widetilde{\phi|_K}$, this is equivalent to $\phi\in K^\perp$. But then $\Phi|_{K^\perp}=1$ implies that $\Phi(\phi^{-1})=1$ for all such $\phi$. Thus, the coefficient of $[G/N,1]_{G/N}$ on the left hand side of (\ref{eqn m equation}) is equal to
\begin{equation*}
   \frac{1}{|G|\cdot|G^*|} \sum_{\substack{K\le G\\ \Phi|_{K^\perp}=1 \\ KN=G}} |K|\,|K^\perp|\, \mu(K,G)\,.
\end{equation*}
On the right hand side of (\ref{eqn m equation}) only the summands with $U=G$ and $\psi=1$ contribute to the coefficient of $[G/N,1]_{G/N}$ and this coefficient evaluates to $m^N_{(G,\Phi)}/|(G/N)^*|$. The result follows.   
\end{proof}

\begin{proposition}\label{prop ind of e}
Let $H$ be a subgroup of $G$ and let $(K,\Psi)\in \calX(H)$. Then
\begin{equation*}
   \ind_H^G(e_{(K,\Psi)}^H) = \frac{|N_G(K,\Psi)|}{|N_H(K,\Psi)|}\cdot e_{(K,\Psi)}^G\,.
\end{equation*}
\end{proposition}

\begin{proof}
This is an immediate consequence of the explicit formula in (\ref{eqn idempotent formula 3}), since $\ind_H^G([L,\phi]_H)=[L,\phi]_G$ for any $(L,\phi)\in \calM(H)$.
\end{proof}

\begin{proposition}\label{prop isom and twist on idempotents}
{\rm (a)} For every isomorphism $f\colon G\myiso G'$ and $(H,\Phi)\in\calX(G)$ one has $\isom_f(e_{(H,\Phi)}^G) = e_{(f(H),\Phi\circ (f|_H)^*)}^{G'}$.

\smallskip
{\rm (b)} For every $g\in G$, $H\le G$, and $(K,\Psi)\in\calX(H)$ one has $\lexp{g}{e_{(K,\Psi)}^H} = e_{(\lexp{g}{K},\lexp{g}{\Psi})}^{\lexp{g}{H}}$.

\smallskip
{\rm (c)} For every $(H,\Phi)\in\calX(G)$ and $\alpha\in G^*$ one has $\tw_\alpha(e_{(H,\Phi)}^G) = \Phi(\alpha|_H)\, e_{(H,\Phi)}^G$ and $\Delta(e_{(H,\Phi)}^G)\cdot_G \tw_\alpha= \Phi(\alpha|_H)\, \Delta(e_{(H,\Phi)}^G)$, with $\Delta$ as in \ref{noth A-fibered Burnside ring}(c).
\end{proposition}

\begin{proof}
All three parts follow immediately from the explicit formulas for the three operations in \ref{noth A-fibered biset functors}(c), the explicit idempotent formula (\ref{eqn idempotent formula 3}), and the formula in (\ref{eqn tensor product formula}).
\end{proof}


\section{Three lemmata}\label{sec two lemmas}

Throughout this section, $G$ and $H$ denote finite groups and $\KK$ a field of characteristic $0$ such that $S^*$ is finite and $\KK$ is a splitting field of $S^*$ for all subquotients $S$ of $G$ and $H$.

\smallskip
For $U\le G\times H$ we set $q(U):=U/(k_1(U)\times k_2(U)$. Thus, $q(U)\cong p_i(U)/k_i(U)$ for $i=1,2$.

\begin{lemma}\label{lem ext and fact}
Let $G$ and $H$ be finite groups and let $k$ be a commutative ring. For $(U,\phi)\in\calM(G\times H)$ with $p_1(U)=G$ and $p_2(U)=H$ the following are equivalent:

\smallskip
{\rm (i)} There exists $\alpha\in G^*$ such that $\alpha|_{k_1(U)}=\phi_1$.

\smallskip
{\rm (ii)} There exists $\beta\in H^*$ such that $\beta|_{k_2(U)}=\phi_2$.

\smallskip
{\rm (iii)} There exists $\psi\in(G\times H)^*$ such that $\psi|_U=\phi$.

\smallskip
{\rm (iv)} In the category $\calC^A$, the morphism $\left[\frac{G\times H}{U,\phi}\right]$ factors through $q(U)$.

\smallskip
{\rm (v)} In the category $\calC^A_k$, the morphism $\left[\frac{G\times H}{U,\phi}\right]$ factors through $q(U)$.
\end{lemma}

\medskip
\begin{proof}
Clearly, (iii) implies (i) and (ii).

\smallskip
We next show that (i) implies (iii). Let $\alpha\in G^*$ be an extension of $\phi_1$. Since $\phi\in U^*$ and $\alpha\times 1\in(G\times\{1\})^*$ coincide on $U\cap(G\times \{1\})=k_1(U)\times\{1\}$ and since $G\times H = (G\times \{1\})U$ with $G\times \{1\}$ normal in $G\times H$, the function $\psi\colon G\times H\to A$, $(g,1)u\mapsto \alpha(g)\phi(u)$ is well-defined and extends $\phi$. It is also a homomorphism, since $G\times \{1\}$ is normal in $G\times H$ and $\alpha\times 1$ is $U$-stable.

\smallskip
Similarly one proves that (ii) implies (iii).

\smallskip
Next we show that (iii) and  implies (iv). Let $\psi=\alpha\times \beta\in (G\times H)^*$ extend $\phi\in U^*$. By (\ref{eqn tensor product formula}), we have
\begin{equation*}
   \left[\frac{G\times H}{U,\phi}\right] = \tw_{\alpha} \cdotG  \left[\frac{G\times H}{U, 1}\right] \cdotH \tw_{\beta} 
\end{equation*}
and the morphism $\left[\frac{G\times H}{U, 1}\right]$ factors through $G/k_1(U)$ by Theorem~\ref{thm can decomp}.

\smallskip
Clearly, (iv) implies (v).

Finally, we show that (v) implies (iii). Assume that $\left[\frac{G\times H}{U,\phi}\right]$ factors in $\calC_k^A$ through $K:=q(U)$ with $K\cong G/k_1(U)$. Then there exist $(V,\Psi)\in \calM(G\times K)$ and $(W,\rho)\in\calM(K\times H)$ such that $\left[\frac{G\times H}{U,\phi}\right]$ occurs with nonzero coefficient in $\left[\frac{G\times K}{V,\psi}\right]\cdot_K\left[\frac{K\times H}{W,\rho}\right]$. By the formula in (\ref{eqn tensor product formula}), this implies that there exists $t\in K$ and $(g,h)\in G\times H$ such that $\psi_2|_{K_t}=\rho_1|_{K_t}$ with $K_t:=k_2(V)\cap \lexp{t}{k_1(W)}$ and
$(U,\phi)=\lexp{(g,h)}{(V*\lexp{(t,1)}{W}, \psi*\lexp{(t,1)}{\rho})}$. Replacing $(V,\psi)$ with $\lexp{(g,1)}{(V,\psi)}$ and $(W,\rho)$ with $\lexp{(t,h)}{(W,\rho)}$, we may assume that there exist $(V,\psi)\in\calM(G\times K)$ and $(W,\rho)\in \calM(K\times H)$ with $\psi_2|_{k_2(V)\cap k_1(W)} = \rho_1|_{k_2(V)\cap k_1(W)}$ and $(V*W,\psi*\rho)=(U,\phi)$. By 
\cite[2.3.22.2]{Bouc2010} one has 
\begin{equation*}
  k_1(V)\le k_1(V*W)=k_1(U)\le p_1(U)\le p_1(V)\,.
\end{equation*}
Since $p_1(U)=G$, this implies that $p_1(V)=G$. Moreover, since $p_1(U)/k_1(U)$ is isomorphic to a subquotient of $p_1(V)/k_1(V)\cong p_2(V)/k_2(V)$, which in turn is a subquotient of $K\cong p_1(U)/k_1(U)$, we otain that $k_1(V)=k_1(U)$, $p_2(V)=K$ and $k_2(V)=\{1\}$. Since $\psi_2=1$, it extends trivially to $K=p_2(V)$. By the first part of the proof (note that $p_1(V)=G$ and $p_2(V)=K$) we also obtain that $\psi$ extends to $\alpha\times 1\in (G\times K)^*$ for some $\alpha\in G^*$. Similarly one shows that $\rho$ extends to $1\times \beta\in (K\times H)^*$ for some $\beta\in H^*$. But then $\phi=\psi*\rho$ is the restriction of $(\alpha\times 1)*(1\times \beta)=\alpha\times\beta$ and (iii) holds.
\end{proof}

\begin{lemma}\label{lem idem and ext}
Let $(U,\phi)\in \calM(G\times H)$. 

\smallskip
{\rm (a)} If $\Phi\in \Hom(H^*,\KK^\times)$ satisfies $\left[\frac{G\times H}{U,\phi}\right]\cdot_H e_{(H,\Phi)}^H\neq 0$ then $p_2(U)=H$ and $\phi_2$ extends to $H^*$.

\smallskip
{\rm (b)} If $\Phi\in \Hom(G^*,\KK^\times)$ satisfies $e_{(G,\Phi)}^G\cdot_G \left[\frac{G\times H}{U,\phi}\right] \neq 0$ then $p_1(U)=G$ and $\phi_1$ extends to $G^*$.
\end{lemma}

\begin{proof}
We only prove Part~(a). Part~(b) is proved similarly.

\smallskip
Since $\left[\frac{G\times H}{U,\phi}\right]= \left[\frac{G\times p_2(U)}{U,\phi}\right] \cdot_{p_2(U)} \res^H_{p_2(U)}$, Proposition~\ref{prop s comp res}(c) implies that $p_2(U)=H$.

We will show that $\phi_2$ extends to $H$ by induction on $|G|$. If $|G|=1$ then $\phi_1$ is trivial, thus extends to $G$, and Lemma~\ref{lem ext and fact} implies that $\phi_2$ extends to $H$. From now on assume that $|G|>1$. 
We distinguish two cases.

\smallskip
{\em Case 1:} $\pi_G\left(\left[\frac{G\times H}{U,\phi}\right]\cdot_H e_{(H,\Phi)}^H\right) = 0$. By Lemma~\ref{lem pi and e}, this implies that $e_G^G\cdot \left(\left[\frac{G\times H}{U,\phi}\right]\cdot_H e_{(H,\Phi)}^H\right)=0$ and therefore, using the primitive idempotents $e_K^G$ of the Burnside ring (see Remark~\ref{rem comparing B and BA}) and \cite[Proposition~2]{CoskunYilmaz2019},  we have
\begin{align*}
  0 &\neq \left[\frac{G\times H}{U,\phi}\right] \cdotH e_{(H,\Phi)}^H 
            = (1-e_G^G)\cdot \left( \left[\frac{G\times H}{U,\phi}\right] \cdotH e_{(H,\Phi)}^H \right) \\
    & = \sum_{\substack{K\in [G\backslash \calS(G)]\\ K<G}} 
                    e_K^G \cdot \left( \left[\frac{G\times H}{U,\phi}\right] \cdotH e_{(H,\Phi)}^H \right)
           = \sum_{\substack{K\in [G\backslash \calS(G)]\\ K<G}}
                       \Delta(e_K^G)\cdotG \left (\left[\frac{G\times H}{U,\phi}\right] \cdotH e_{(H,\Phi)}^H \right)\,.
\end{align*}
Moreover, using the explicit formula for $e_K^G$ (in the special case that $A$ is trivial), we have
\begin{align*}
    0 & \neq   \left[\frac{G\times H}{U,\phi}\right] \cdotH e_{(H,\Phi)}^H
              \in \sum_{K<G}\KK\cdot\left[\frac{G\times G}{\Delta(K),1}\right] \cdotG 
                          \left( \left[\frac{G\times H}{U,\phi}\right] \cdotH e_{(H,\Phi)}^H \right) \\
       & = \sum_{K<G} \KK \cdot \left( \left[\frac{G\times G}{\Delta(K),1}\right] \cdotG \left[\frac{G\times H}{U,\phi}\right] \right)
                                    \cdotH e_{(H,\Phi)}^H
              = \sum_{K<G} \KK \cdot \left[\frac{G\times H}{\Delta(K)*U,1*\phi}\right] \cdotH e_{(H,\Phi)}^H\,.
\end{align*}
Therefore, there exists $K<G$ such that $\left[\frac{G\times H}{\Delta(K)*U,1*\phi}\right] \cdot_H e_{(H,\Phi)}^H\neq 0$. Note that $p_1(\Delta(K)*U)\le K$, so that we can decompose
\begin{equation*}
  \left[\frac{G\times H}{\Delta(K)*U,1*\phi}\right] = \ind_K^G \cdotK \left[\frac{K\times H}{\Delta(K)*U,1*\phi}\right]\,.
\end{equation*}
Thus,
\begin{equation}\label{eqn smaller situation}
  \left[\frac{K\times H}{\Delta(K)*U,1*\phi}\right]\cdotH e_{(H,\Phi)}^H \neq 0\,.
\end{equation}
By the first part of the proof this implies that $p_2(\Delta(K)*U)=H$. Moreover, it is straightforward to verify that $k_2(\Delta(K)*U)=k_2(U)$ and that $(1*\phi)_2=\phi_2\in k_2(U)^*$. Since $K<G$, the inductive hypothesis applied to (\ref{eqn smaller situation}) yields that $\phi_2$ extends to $H$.

\smallskip
{\em Case 2:} $\pi_G\!\left(\left[\frac{G\times H}{U,\phi}\right]\cdot_H e_{(H,\Phi)}^H\right) \neq 0$. The explicit formula for $e_{(H,\Phi)}^H$ in (\ref{eqn idempotent formula 3}) implies that
\begin{equation*}
   \sum_{\substack{K\le H\\ \Phi|_{K^\perp}=1}}\ \sum_{\psi\in H^*} |K|\, \mu(K,H)\, \Phi(\psi^{-1})\,
   \pi_G\!\left( \left[\frac{G\times H}{U,\phi}\right] \cdotH [K,\psi|_K]_H \right) \neq 0\,.
\end{equation*}
Thus, there exists $K\le H$ and $\psi\in H^*$ such that $\pi_G \left( \left[\frac{G\times H}{U,\phi}\right] \cdot_H [K,\psi|_K]_H \right) \neq 0$. This implies that $U*K=G$ and $\phi_2|_{k_2(U)\cap K} = \psi|_{k_2(U)\cap K}$. Since $\phi_2$ is stable under $p_2(U)=H$ and $\phi_2$ and $\psi$ coincide on $k_2(U)\cap K$, there exists an extension $\beta\in (k_2(U)K)^*$ of $\phi_2$ and $\psi$. Moreover, $U*K=G$ implies $k_2(U)K=H$. In fact, if $h\in H=p_2(U)$ then there exists $g\in G$ with $(g,h)\in U$. Since $g\in G=U*K$, there exists $k\in K$ such that $(g,k)\in U$. But then $hk^{-1}\in k_2(U)$ and $h\in k_2(U)K$. This completes the proof of the lemma.
\end{proof}

For $U\le G\times H$ we denote by $\eta_U\colon p_2(U)/k_2(U)\myiso p_1(U)/k_1(U)$ the isomorphism induced by $U$ (see \cite[Lemma 2.3.25]{Bouc2010}).

\begin{lemma}\label{lem double prod}
Let $(U,\phi)\in\calM(G\times H)$, $\Phi\in\Hom(G^*,\KK^\times)$ and $\Psi\in\Hom(H^*,\KK^\times)$. Then
\begin{equation*}
   e_{(G,\Phi)}^G\cdot\left( \left[\frac{G\times H}{U,\phi}\right] \cdotH e_{(H,\Psi)}^H \right)= 0
\end{equation*}
unless $p_1(U)=G$, $p_2(U)=H$, $\phi$ has an extension $\alpha\times \beta\in (G\times H)^*$, $\Phi_{k_1(U)}=\Psi_{k_2(U)}\circ\eta_U^*$, and $m_{(H,\Phi)}^{k_2(U)}\neq 0$, in which case one has
\begin{equation*}
   e_{(G,\Phi)}^G\cdot\left( \left[\frac{G\times H}{U,\phi}\right] \cdotH e_{(H,\Psi)}^H \right)
   = \Phi(\alpha)\, \Psi(\beta)\, m_{(H,\Psi)}^{k_2(U)}\cdot e_{(G,\Phi)}^G\,.
\end{equation*}
\end{lemma}

\begin{proof}
Assume that 
\begin{equation}\label{eqn neq 0}
e_{(G,\Phi)}^G\cdot \left( \left[\frac{G\times H}{U,\phi}\right] \cdotH e_{(H,\Psi)}^H \right) \neq 0\,
\end{equation}
By Lemma~\ref{lem idem and ext}, $p_2(U)=H$ and $\phi_2$ extends to $H$. Assume $p_1(U)<G$. Then $\left[\frac{G\times H}{U,\phi}\right]\cdot e_{(H,\Psi)}^H$ is in the $\KK$-span of elements of the form $[U*L,\psi]_G$, with $L\le H$ and $\psi\in(U*L)^*$. Since $U*L\le p_1(U)<G$ we obtain $m_G(e_{(G,\Phi)}\cdot [U*L,\psi]_G) = m_G(e_{(G,\Phi)})\cdot m_G([U*L,\psi]_G) = 0$, because the first factor vanishes in the components indexed by proper subgroups of $G$ and the second factor vanishes in the $G$-component. By the injectivity of $m_G$ this implies that the element in (\ref{eqn neq 0}) vanishes, a contradiction. Thus, $p_1(U)=G$. By Lemma~\ref{lem ext and fact}, $\phi$ has an extension $\alpha\times \beta\in (G\times H)^*$. Further,
\begin{align*}
   e_{(G,\Phi)}^G\cdot \left( \left[\frac{G\times H}{U,\phi}\right] \cdotH e_{(H,\Psi)}^H \right)
   & = e_{(G,\Phi)}^G \cdot \left(\tw_\alpha \cdotG \left[\frac{G\times H}{U,1}\right]\cdotH \tw_\beta \cdotH e_{(H,\Psi)}^H\right)\\
   & = \Delta(e_{(G,\Phi)}^G) \cdotG \tw_\alpha \cdotG \left[\frac{G\times H}{U,1}\right]\cdotH \tw_\beta \cdotH e_{(H,\Psi)}^H\,,
\end{align*}
where the last equation follows from \cite[Proposition~2]{CoskunYilmaz2019}. By Proposition~\ref{prop isom and twist on idempotents}(c) we have $\tw_\beta\cdot_H e_{(H,\Psi)}^H= \tw_\beta(e_{(H,\Psi)}^H)= \Psi(\beta)\cdot e_{(H,\Psi)}^H$ and $\Delta(e_{(G,\Phi)}^G)\cdot_G \tw_\alpha=\Phi(\alpha)\cdot \Delta(e_{(G,\Phi)}^G)$. Thus, using again \cite[Proposition~2]{CoskunYilmaz2019} and Theorem~\ref{thm can decomp},
\begin{align*}
   & e_{(G,\Phi)}^G\cdot \left( \left[\frac{G\times H}{U,\phi}\right] \cdotH e_{(H,\Psi)}^H \right) 
        = \Phi(\alpha)\Psi(\beta)\cdot \Delta(e_{(G,\Phi)}^G)\cdotG \left[\frac{G\times H}{U,1}\right] \cdotH e_{(H,\Psi)}^H\\
   & = \Phi(\alpha)\, \Psi(\beta)\, e_{(G,\Phi)}\cdot \left( \infl_{G/k_1(U)}^G\cdott{G/k_1(U)} \left[\frac{G\times H}{\Ubar,1}\right]  
                  \cdott{H/k_2(U)} \defl^H_{H/k_2(U)} \cdotH e_{(H,\Psi)}^H \right)\\
   & =  \Phi(\alpha)\, \Psi(\beta)\, e_{(G,\Phi)}\cdot \left( \infl_{G/k_1(U)}^G\cdott{G/k_1(U)} \isom_{\eta_U}  
                  \cdott{H/k_2(U)} \defl^H_{H/k_2(U)} \cdotH e_{(H,\Psi)}^H \right)\,.
\end{align*}
Using Propositions~\ref{prop m number}(a), \ref{prop isom and twist on idempotents}(a) and \ref{prop s comp inf}(b), we obtain 
\begin{equation}\label{eqn still neq 0}
   e_{(G,\Phi)}^G\cdot \left( \left[\frac{G\times H}{U,\phi}\right] \cdotH e_{(H,\Psi)}^H \right)  = \Phi(\alpha)\, \Psi(\beta)\, m_{(H,\Psi)}^{k_2(U)}\, \sum_{(K,\Theta)} e_{(G,\Phi)}^G\cdot e_{(K,\Theta)}^G\,,
\end{equation}
where the sum runs over those $(K,\Theta)\in[G\backslash \calX(G)]$ satisfying $(Kk_1(U)/k_1(U),\Theta_{k_1(U)})=_{G/k_1(U)} (G/k_1(U),\Psi_{k_2(U)}\circ \eta_U^*)$. Since the term in (\ref{eqn still neq 0}) is nonzero, one of these pairs $(K,\Theta)$ must be $G$-conjugate, and then also equal, to $(G,\Phi)$. This implies the result.
\end{proof}



\section{The constant $m_{(G,\Phi)}^N$}\label{sec m}

Throughout this section, $G$ and $H$ denote finite groups and $\KK$ denotes a field of characteristic $0$ such that for any subquotients $S$ of $G$ and $T$ of $H$ the groups $S^*$ and $T^*$ are finite and $\KK$ is a splitting field for $S^*$ and $T^*$.

\smallskip
In this section we prove the crucial Proposition~\ref{prop G/M isomorphic G/N} stating that $m_{(G,\Phi)}^M=m_{(G,\Phi)}^N$ if $(G/M,\Phi_M)\cong(G/N,\Phi_N)$ (see Definition~\ref{def BA-pair}(b)).

\begin{proposition}\label{prop trans m formula}
Let $N$ and $M$ be normal subgroups of $G$ with $N\le M$ and let $\Phi\in\Hom(G^*,\KK^\times)$. Then
\begin{equation*}
   m_{(G,\Phi)}^M = m_{(G,\Phi)}^N \cdot m_{(G/N,\Phi_N)}^{M/N}\,.
\end{equation*}
\end{proposition}

\begin{proof}
This follows immediately from Proposition~\ref{prop m number}(a) and applying $\isom_f\circ\defl^{G/N}_{(G/N)/(M/N)}\circ\defl^G_{G/N}=\defl^G_{G/M}$ to $e_{(G,\Phi)}^G$, where $f\colon (G/N)/(M/N)\to G/M$ is the canonical isomorphism.
\end{proof}

\begin{lemma}\label{lem composing with Phi}
Let $f_1,f_2\colon G\to H$ be group homomorphisms, let $\Phi\in\Hom(G^*,\KK^\times)$, and let $K\le G$ be such that $\Phi|_{K^\perp}=1$ and $f_1|_K=f_2|_K$. Then $\Phi\circ f_1^*=\Phi\circ f_2^*\in\Hom(H^*,\KK^\times)$.
\end{lemma}

\begin{proof}
Let $\lambda\in H^*$. Then $(\Phi\circ f_1^*)(\lambda)=(\Phi\circ f_2^*)(\lambda)$ if and only if $\Phi(\lambda\circ f_1)=\Phi(\lambda\circ f_2)$ which in turn is equivalent to $\Phi((\lambda\circ f_1)\cdot(\lambda\circ f_2)^{-1})=1$. But, since $f_1|_K=f_2|_K$, we have $(\lambda\circ f_1)\cdot(\lambda\circ f_2)^{-1}\in K^\perp$ and since $\Phi|_{K^\perp}=1$ the result follows.
\end{proof}

\begin{proposition}\label{prop m for M and N}
For $\Phi\in \Hom(G^*,\KK^\times)$ and normal subgroups $M$ and $N$ of $G$ one has
\begin{equation*}
  m_{(G,\Phi)}^M = \frac{1}{|G|\cdot|G^*|} \sum_{\substack{ K\le G\\ KN=KM=G \\ \Phi|_{K^\perp}=1}}
      |K|\, \mu(K,G)\, |\Sigma^K_{M,N}|\, m_{(G/N,\Phi_N)}^{(K\cap M)N/N}\,,
\end{equation*}
where $\Sigma^K_{M,N}$ is the set of elements $\phi\in G^*$ such that $\phi|_{K\cap M\cap N}=1$ and such that there exists $\psi\in ((G/M)\times (G/N))^*$ with $\psi(kM,kN)=\phi(k)$ for all $k\in K$.
\end{proposition}

\begin{proof}
Consider the element
\begin{equation*}
     v:= e_{(G/M,\Phi_M)}^{G/M} \cdot 
     \left( \defl^G_{G/M} \cdotG \Delta(e_{(G,\Phi)}^G) \cdotG \infl_{G/N}^G \cdott{G/N} e_{(G/N,\Phi_N)}^{G/N} \right) 
     \in B_\KK^A(G/M)\,.
\end{equation*}
Then, on the one hand,
\begin{equation}\label{eqn v 1}
   v=m_{(G,\Phi)}^M e_{(G/M,\Phi_M)}^{G/M}\in B^A_\KK(G/M)\,.
\end{equation}
In fact, by \cite[Proposition~2]{CoskunYilmaz2019} and Proposition~\ref{prop s comp inf}(b),
\begin{equation*}
   \Delta(e_{(G,\Phi)}^G)\cdotG\infl_{G/N}^G\cdott{G/N} e_{(G/N,\Phi_N)}^{G/N} 
   = e_{(G,\Phi)}^G \cdot (\infl_{G/N}^G\cdott{G/N} e_{(G/N,\Phi_N)}^{G/N}) = e_{(G,\Phi)}^G\,,
\end{equation*}
and then Proposition~\ref{prop m number}(a) implies (\ref{eqn v 1}). On the other hand, using the formula in (\ref{eqn idempotent formula 3}) for $e_{(G,\Phi)}^G$ we obtain
\begin{equation}\label{eqn v 2}
   v=\frac{1}{|G|\cdot|G^*|} \sum_{\substack{K\le G \\ \Phi|_{K^\perp}=1}}\, \sum_{\phi \in G^*} 
   |K|\, \mu(K,G)\, \Phi(\phi^{-1})\,
    e_{(G/M,\Phi_M)}^{G/M} \cdot \left( x_{K,\phi} \cdott{G/N} e_{(G/N,\Phi_N)}^{G/N} \right)
\end{equation}
with
\begin{equation*}
  x_{K,\phi}:= \defl^G_{G/M} \cdotG \Delta([K,\phi|_K]_G) \cdotG \infl_{G/N}^G 
  = \begin{cases} \left[\frac{G/M\times G/N}{\Delta_{M,N}^K,\phibar}\right], & \text{if $\phi|_{K\cap M\cap N}=1$,}\\
                                    0, & \text{otherwise,}  
             \end{cases}
\end{equation*}
for $K\le G$ with $\Phi|_{K^\perp}=1$, by the formula in (\ref{eqn tensor product formula}),
where, in the first case, $\Delta_{M,N}^K:=\{(kM,kN)\mid k\in K\}$  and $\phibar((kM,kN):=\phi(k)$ for $k\in K$. 
Note that $A_K:=k_1(\Delta_{M,N}^K)=(K\cap N)M/M$, $B_K:=k_2(\Delta_{M,N}^K)=(K\cap M)N/N$, $p_1(\Delta_{M,N}^K)= KM/M$, and $p_2(\Delta_{M,N}^K)=KN/N$.
Lemma~\ref{lem double prod} implies that, if the term $e_{(G/M,\Phi_M)}^{G/M}\cdot (x_{K,\phi} \cdott{G/N} e_{(G/N,\Phi_N)}^{G/N})$ in (\ref{eqn v 2}) is nonzero then $KM=G=KN$, $\phi|_{K\cap M\cap N}=1$, and $\phibar$ extends to some $\psi=\alpha\times \beta\in((G/M)\times (G/N))^*$.
Then the formula in Lemma~\ref{lem double prod} yields
\begin{align}\notag
  v & =\frac{1}{|G|\cdot|G^*|} \sum_{\substack{K\le G\\ KM= G=KN\\ \Phi|_{K^\perp}=1}}\
           \sum_{\phi\in \Sigma^K_{M,N}} |K|\,\mu(K,G)\,
           \Phi(\phi^{-1})\, \Phi_M(\alpha_{K,\phi})\, \Phi_N(\beta_{K,\phi})\, m_{(G/N,\Phi_N)}^{B_K} \cdot \\ \label{eqn v 3}
   & \qquad    \cdot e_{(G/M,\Phi_M)}^{G/M}\,,
\end{align}
where, for $K$ and $\phi$ as above, $\alpha_{K,\phi}\in (G/M)^*$ and $\beta_{K,\psi}\in (G/N)^*$ are chosen such that $\alpha_{K,\phi}\times \beta_{K,\phi}$ is an extension of $\phibar$. Denoting the inflations of $\alpha_{K,\psi}$ and $\beta_{K,\psi}$ to $G$ by $\alphatilde_{K,\psi}\in$ and $\betatilde_{K,\psi}$, we have
\begin{equation*}
   \Phi(\phi^{-1})\Phi_M(\alpha_{K,\phi})\Phi_N(\beta_{K,\phi}) = \Phi(\phi^{-1}\alphatilde_{K,\phi}\betatilde_{K,\phi})=1
\end{equation*}
since $\alphatilde_{K,\phi}(k)\betatilde_{K,\phi}(k) = \alpha_{K,\phi}(kM)\beta_{K,\phi}(kN) = \phibar(kM,kN)=\phi(k)$ for all $k\in K$ and $\Phi|_{K^\perp}=1$. Thus,
\begin{equation}\label{eqn v 4}
   v = \frac{1}{|G|\cdot|G^*|} \sum_{\substack{K\le G\\ KM= G=KN\\ \Phi|_{K^\perp}=1}}
   |\Sigma^K_{M,N}|\, |K|\,\mu(K,G)\, m_{(G/N,\Phi_N)}^{B_K}\,e_{(G/M,\Phi_M)}^{G/M}\,.
\end{equation}
Comparing Equations~(\ref{eqn v 1}) and (\ref{eqn v 4}), the formula for $m_{(G,\Phi)}^M$ follows.
\end{proof}

\begin{proposition}\label{prop G/M isomorphic G/N}
Let $M$ and $N$ be normal subgroups of $G$ such that there exists an isomorphism $f\colon G/N\myiso G/M$ and let $\Phi\in \Hom(G^*,\KK^\times)$ be such that $\Phi_N\circ f^*=\Phi_M \in\Hom((G/M)^*,\KK^\times)$. Then one has $m_{(G,\Phi)}^M=m_{(G,\Phi)}^N$.
\end{proposition}

\begin{proof}
We proceed by induction on $|G|$. If $|G|=1$ the result is clearly true. So assume that $|G|>1$. If $M=N$ is the trivial subgroup of $G$ then the result holds for trivial reasons.  So assume that $M$ and $N$ are not trivial.
By Proposition~\ref{prop m for M and N} one has
\begin{equation*}
   m_{(G,\Phi)}^M = \frac{1}{|G|\cdot|G^*|} \sum_{\substack{ K\le G\\ KN=KM=G \\ \Phi|_{K^\perp}=1}}
      |K|\, \mu(K,G)\, |\Sigma^K_{M,N}|\, m_{(G/N,\Phi_N)}^{(K\cap M)N/N}
\end{equation*}
and
\begin{equation*}
     m_{(G,\Phi)}^N = \frac{1}{|G|\cdot|G^*|} \sum_{\substack{ K\le G\\ KN=KM=G \\ \Phi|_{K^\perp}=1}}
      |K|\, \mu(K,G)\, |\Sigma^K_{N,M}|\, m_{(G/M,\Phi_M)}^{(K\cap N)M/M}\,.
\end{equation*}
We will show that these sums coincide by comparing them summand by summand. Since $\Sigma^K_{M,N}=\Sigma^K_{N,M}$, it suffices to show that $m_{(G/N,\Phi_N)}^{(K\cap M)N/N}=m_{(G/M,\Phi_M)}^{(K\cap N)M/M}$. By Proposition~\ref{prop m number}(b), for any $X\trianglelefteq G/N$ we have
\begin{equation*}
   m_{(G/N,\Phi_N)}^X = \frac{|((G/N)/X)^*|}{|G/N|\cdot|(G/N)^*|} 
   \sum_{\substack{L\le G/N\\ LX=G/N\\ (\Phi_N)|_{L^\perp} = 1}} |L|\, |L^\perp||\, \mu(L,G/N)
\end{equation*}
and
\begin{equation*}
   m_{(G/M,\Phi_M)}^{ f(X)} = \frac{|((G/M)/f(X))^*|}{|G/M|\cdot|(G/M)^*|}
   \sum_{\substack{K\le G/N\\ K f(X)=G/N\\ \Phi_M|_{K^\perp}=1}} |K|\, |K^\perp|\, \mu(K,G/M)\,.
\end{equation*}
Note that the summand for $L$ in the first sum is equal to the summand for $K=f(L)$ in the second sum. Thus, with $X=(K\cap M)N/N$, we obtain
\begin{equation}\label{eqn m iterated}
   m_{(G/N,\Phi_N)}^{(K\cap M)N/N} = m_{(G/M,\Phi_M)}^{f((K\cap M)N/N)}
\end{equation}
for any $K\le G$ with $KM=G=KN$ and $\Phi|_{K^\perp}=1$. It now suffices to find an isomorphism $f_1\colon (G/M)/f((K\cap M)N/N)\myiso (G/M)/((K\cap N)M/M)$ such that $(\Phi_M)_{f((K\cap M)N/N)} \circ f_1^*= (\Phi_M)_{(K\cap N)M/M}$, since then, by induction, we have $m_{(G/M,\Phi_M)}^{f((K\cap M)N/N)}=m_{(G/M,\Phi_M)}^{(K\cap N)M/M}$ and together with Equation~(\ref{eqn m iterated}) this implies that desired equation. 
We define $f_1:= \eta\circ\fbar^{-1}$, where $\fbar\colon (G/N)/((K\cap M)N/N)\myiso (G/M)/f((K\cap M)N/N)$ is induced by $f$ and $\eta:=\eta_{\Delta^K_{M,N}}\colon (G/N)/((K\cap M)N/N)\myiso (G/M)/((K\cap N)M/M)$ is induced by $\Delta^K_{M,N}:=\{(kM,kN)\mid k\in K\}$, see \cite[Lemma 2.3.25]{Bouc2010}. It remains to prove that 
\begin{equation}\label{eqn Phi comp}
  \Phi\circ \nu_M^*\circ \nu_{f((K\cap M)N/N)}^* \circ f_1^* = \Phi\circ \nu_M^*\circ \nu_{((K\cap N)M/M)}^*\,,
\end{equation}
where $\nu_M\colon G\to G/M$, $\nu_{f(K\cap M)N/N}\colon G/M \to (G/M)/f((K\cap M)N/N)$, and $\nu_{(K\cap N)M/M}\colon G/M\to (G/M)/((K\cap N)M/M)$ denote the natural epimorphisms.
But $f_1^*=(\fbar^{-1})^*\circ\eta^*$, $\nu_{f((K\cap M)N/N)}^*\circ (\fbar^{-1})^* = (f^{-1})^*\circ \nu_{(K\cap N)M/M}^*$, and $\Phi\circ\nu_M^*\circ (f^{-1})^*=\Phi_M\circ (f^{-1})^* = \Phi_N=\Phi\circ \nu_N^*$. Thus, the left hand side of Equation~(\ref{eqn Phi comp}) is equal to $\Phi\circ\nu_N^*\circ \nu_{(K\cap M)N/N}^*\circ\eta^*$. By Lemma~\ref{lem composing with Phi} and since $\Phi|_{K^\perp}=1$, it now suffices to show that
\begin{equation*}
   (\eta\circ\nu_{(K\cap M)N/N}\circ\nu_N)|_K = (\nu_{(K\cap M)N/N}\circ\nu_N)|_K\,.
\end{equation*}
But this follows from $(kM,kN)\in\Delta_{M,N}^K$ for $k\in K$, and the proof is complete.
\end{proof}


\section{$B^A$-pairs and the subfunctors $E_{(G,\Phi)}$ of $B_{\KK}^A$}\label{sec BA-pairs and the subfunctors E}

Throughout this section we assume that $G^*$ is finite for all finite groups $G$, and we assume that $\KK$ is a field of characteristic $0$ which is a splitting field of $\KK G^*$ for all finite groups $G$. The latter is equivalent to requiring that, for any torsion element $a$ of $A$, the field $\KK$ has a root of unity whose order equals the order of $a$.

\smallskip
In this section we introduce the important subfunctors $E_{(G,\Phi)}$ of $B_\KK^A$ and study their properties.

\begin{definition}\label{def E}
For any finite group $G$ and $\Phi\in\Hom(G^*,\KK^\times)$ we denote by $E_{(G,\Phi)}$ the subfunctor of $B_\KK^A$ generated by $e_{(G,\Phi)}^G$. In other words, for each finite group $H$, one has 
\begin{equation*}
   E_{(G,\Phi)}(H)=\{x\cdot_G e_{(G,\Phi)}^G\mid x\in B_\KK^A(H,G)\}\,.
\end{equation*}
\end{definition}

\begin{proposition}\label{prop E and m}
For any finite group $G$ and $\Phi\in\Hom(G^*,\KK^\times)$, the following are equivalent:

\smallskip
{\rm (i)} If $H$ is a finite group with $E_{(G,\Phi)}(H) \neq \{0\}$ then $|G|\le |H|$.

\smallskip
{\rm (ii)} If $H$ is a finite group with $E_{(G,\Phi)}(H) \neq \{0\}$ then $G$ is isomorphic to a subquotient of $H$.

\smallskip
{\rm (iii)} For all $\{1\}\neq N\trianglelefteq G$ one has $m_{(G,\Phi)}^N=0$.

\smallskip
{\rm (iv)} For all $\{1\}\neq N\trianglelefteq G$ one has $\defl^G_{G/N}(e_{(G,\Phi)}^G) = 0$.
\end{proposition}

\begin{proof}
That (ii) implies (i) and that (i) implies (iv) follows from the definitions. Moreover, that (iii) and (iv) are equivalent follows from Proposition~\ref{prop m number}(a). So, it suffices to prove that (iv) implies (ii). 

Assume that (iv) holds and let $H$ be a finite group with $E_{(G,\Phi)}(H)\neq \{0\}$. By the definition of $E_{(G,\Phi)}$ this implies that there exists $(U,\phi)\in\calM(H\times G)$ such that $\left[\frac{H\times G}{U,\phi}\right]\cdot_G e_{(G,\Phi)}^G \neq 0$. The canonical decomposition of $\left[\frac{H\times G}{U,\phi}\right]$ from Theorem~\ref{thm can decomp} implies that \begin{equation*}
   \left[\frac{(P/K)\times (Q/L)}{\Ubar,\phibar}\right] \cdott{p_2(U)/\ker(\phi_2)} 
   \defl^{Q}_{Q/L} \cdott{Q}\res^G_{Q} \cdotG e_{(G,\Phi)}^G \neq 0\,,
\end{equation*}
with $P:=p_1(U)$, $K:=\ker(\phi_1)$, $Q:=p_2(U)$, $L:=\ker(\phi_2)$, $\Ubar$ corresponding to $U/(K\times L)$ via the canonical isomorphism $(P\times Q)/(K\times L)\cong (P/K)\times (Q/L)$, and $\phibar\in (\Ubar)^*$ induced by $\phi$.
Proposition~\ref{prop s comp res}(c) implies that $G=Q$, and then Proposition~\ref{prop m number}(a) implies that $L=\{1\}$. Thus, $p_1(\Ubar)=P/K$, $p_2(\Ubar)=G$, and
\begin{equation*}
   \left[\frac{P/K\times G}{\Ubar,\phibar}\right] \cdotG e_{(G,\Phi)}^G\neq 0\,.
\end{equation*}
Lemma~\ref{lem idem and ext} implies that $(\phibar)_2$ extends to $G$ and Lemma~\ref{lem ext and fact} implies that $\phibar$ extends to some $\alpha\times \beta \in ((P/K)\times G)^*$ with $\alpha\in (P/K)^*$ and $\beta\in G^*$, since $p_1(\Ubar)=P/K$ and $p_2(\Ubar)=G$. Further, we have
\begin{equation*}
  0\neq \left[\frac{(P/K)\times G}{\Ubar,\phibar}\right] \cdotG e_{(G,\Phi)}^G = 
     \tw_{\alpha}\cdott{P/K} \left[\frac{(P/K)\times G}{\Ubar,1}\right]\cdotG 
     \tw_\beta \cdotG e_{(G,\Phi)}^G
\end{equation*}
with $\tw_\beta \cdotG e_{(G,\Phi)}^G = \Phi(\beta) e_{(G,\Phi)}^G$ by Proposition~\ref{prop isom and twist on idempotents}(c). Thus, using the canonical decomposition of $\left[\frac{(P/K)\times G}{\Ubar,1}\right]$ as in Theorem~\ref{thm can decomp} we obtain $\defl^G_{G/k_2(\Ubar)} \cdot_G e_{(G,\Phi)}^G\neq 0$. Proposition~\ref{prop m number}(a) implies that $k_2(\Ubar)=\{1\}$. But this implies that $G\cong G/\{1\}\cong p_2(\Ubar)/k_2(\Ubar)\cong p_1(\Ubar)/k_1(\Ubar)$ is isomorphic to a subquotient of $p_1(\Ubar)$, which is isomorphic to a subquotient of $H$.
\end{proof}

Note that by the formula for $m_{(G,\Phi)}^N$ in Proposition~\ref{prop m number}, the condition in Proposition~\ref{prop E and m}(iii) is independent of the choice of $\KK$ as long as $\KK$ has enough roots of unity.

\begin{definition}\label{def BA-pair}
Let $G$ and $H$ be finite groups and let $\Phi\in\Hom(G^*,\KK^\times)$ and $\Psi\in\Hom(H^*,\KK^\times)$.

\smallskip
(a) The pair $(G,\Phi)$ is called a {\em $B^A$-pair} if the equivalent conditions in Proposition~\ref{prop E and m} are satisfied.

\smallskip
(b) We call $(G,\Phi)$ and $(H,\Psi)$ {\em isomorphic} and write $(G,\Phi)\cong (H,\Psi)$ if there exists an isomorphism $f\colon H\myiso G$ such that $\Psi\circ f^*=\Phi$. Note that if $(G,\Phi)\cong(H,\Psi)$ then $E_{(G,\Phi)}=E_{(H,\Psi)}$. In fact, if $f\colon H\myiso G$ satisfies $\Psi\circ f^*=\Phi$ then $\isom_f(e_{(H,\Psi)}^H)=e_{(G,\Phi)}^G$, by Proposition~\ref{prop isom and twist on idempotents}(a).

\smallskip
(c) We write $(H,\Psi)\preccurlyeq(G,\Phi)$ if there exists a normal subgroup $N$ of $G$ such that $(H,\Psi)\cong (G/N,\Phi_N)$. The relation $\preccurlyeq$ is reflexive and transitive. Moreover, if $(H,\Psi)\preccurlyeq (G,\Phi)$ and $(G,\Phi)\preccurlyeq(H,\Psi)$ then $(G,\Phi)\cong(H,\Psi)$. Thus, $\preccurlyeq$ induces a partial order on the set of isomorphism classes $[G,\Phi]$ of pairs $(G,\Phi)$, where $G$ is a finite group and $\Phi\in\Hom(G,\KK^\times)$. We denote this relation again by $\preccurlyeq$.
This partial order restricts to a partial order on the set $\calB^A$ of isomorphism classes of $B^A$-pairs.
\end{definition}

\begin{proposition}\label{prop E in E}
Let $G$ and $H$ be finite groups and let $\Phi\in\Hom(G^*,\KK^\times)$ and $\Psi\in\Hom(H^*,\KK^\times)$.

\smallskip
{\rm (a)} If $(H,\Psi)\preccurlyeq(G,\Phi)$ then $E_{(G,\Phi)}\subseteq E_{(H,\Psi)}$.

\smallskip
{\rm (b)} If $(H,\Psi)$ is a $B^A$-pair and $E_{(G,\Phi)}\subseteq E_{(H,\Psi)}$ then $(H,\Psi)\preccurlyeq (G,\Phi)$.
\end{proposition}

\begin{proof}
(a) Let $N\trianglelefteq G$ and let $f\colon H\myiso G/N$ be an isomorphism with $\Psi\circ f^* = \Phi_N$. Then, by Proposition~\ref{prop isom and twist on idempotents}(a), Proposition~\ref{prop s comp inf}(b) and \cite[Proposition~2]{CoskunYilmaz2019}, we have
\begin{equation*}
  e_{(G,\Phi)}^G =  e_{(G,\Phi)}^G \cdot ( \infl_{G/N}^G \cdott{G/N}\isom_f \cdotH e_{(H,\Psi)}^H)  \in E_{(H,\Psi)}(G)\,,
\end{equation*}
so that $E_{(G,\Phi)}\subseteq E_{(H,\Psi)}$.

\smallskip
(b) Since $E_{(G,\Phi)}\subseteq E_{(H,\Psi)}$, we have $e_{(G,\Phi)}^G\in E_{(H,\Psi)}(G)$ and there exists $(U,\phi)\in\calM(G\times H)$ such that
\begin{equation}\label{eqn E in E nonzero}
   0 \neq e_{(G,\Phi)}^G \cdot \left(\left[\frac{G\times H}{U,\phi}\right] \cdotH e_{(H,\Psi)}^H \right)\,.
\end{equation}
Lemma~\ref{lem double prod} implies that $p_1(U)=G$, $p_2(U)=H$, $\phi$ extends to some $\alpha\times \beta\in(G\times H)^*$, $\Phi_{k_1(U)}=\Psi_{k_2(U)}\circ\eta_U^*$, and $m_{(H,\Psi)}^{k_2(U)}\neq 0$.
Since $(H,\Psi)$ is a $B^A$-pair, we obtain $k_2(U)=\{1\}$ and $\Phi_{k_1(U)}=\Psi\circ\eta_U^*$.
Thus, $\eta_U$ is an isomorphism $H\myiso G/k_1(U)$ with $\Phi_{k_1(U)} = \Psi\circ \eta_U^*$, so that $(H,\Psi)\preccurlyeq (G,\Phi)$.
%
\end{proof}


\section{Subfunctors of $B_\KK^A$}\label{sec subfunctors of BA}

We keep the assumptions on $A$ and $\KK$ from Section~\ref{sec BA-pairs and the subfunctors E}.
In this section we prove one of our main results, Theorem~\ref{thm lattice iso}, which describes the lattice of subfunctors of $B_\KK^A$ in terms of the poset $(\calB^A, \preccurlyeq)$.

\begin{definition}\label{def minimal group}
Let $k$ be a commutative ring and $F\in\calF_k^A$. A finite group $G$ is called a {\em minimal group} for $F$ if $G$ is a group of minimal order with $F(G)\neq \{0\}$.
\end{definition}

For any finite group $G$, the group $\Aut(G)$ acts on $\calX(G)$ via $\lexp{f}{(K,\Psi)}:=(f(K),\Psi\circ (f|_K)^*)$. We will denote by $\calXhat(G)\subseteq \calX(G)$ the set of those pairs $(K,\Psi)$ with $K=G$. Note that $\calXhat(G)$ is $\Aut(G)$-invariant and that $G$ acts trivially by conjugation on $\calXhat(G)$, so that $\calXhat(G)$ can be viewed as an $\Out(G)$-set.

\begin{proposition}\label{prop min pairs are BA}
Let $F$ be a subfunctor of $B_\KK^A$ in $\calF_\KK^A$.

\smallskip
{\rm (a)} For each finite group $G$ one has 
\begin{equation*}
   F(G)=\bigoplus_{(K,\Psi)\in[G\backslash\calX_F(G)]} \KK \cdot e_{(K,\Psi)}^G\,,
\end{equation*}
where $\calX_F(G):=\{(K,\Psi)\in \calX(G)\mid e_{(K,\Psi)}^G\in F(G)\}$.

\smallskip
{\rm (b)} For any finite group $G$, the set $\calX_F(G)$ is invariant under the action of $\Aut(G)$. 

\smallskip
{\rm (c)} Suppose that $H$ is a minimal group for $F$. Then $\calX_F(H)$ contains only elements from $\calXhat(H)$. Moreover, each $(H,\Psi)\in\calX_F(H)$ is a $B^A$-pair and one has $E_{(H,\Psi)}\subseteq F$. 
\end{proposition}

\begin{proof}
(a) For all $a\in F(G)$ and $x\in B_\KK^A(G)$, \cite[Proposition~2]{CoskunYilmaz2019} implies $x\cdot a= \Delta(x)\cdot_G a\in F(G)$. Thus, $F(G)$ is an ideal of $B_\KK^A(G)$. Since the elements $e_{(K,\Theta)}^G$ with $(K,\Theta)\in[G\backslash\calX(G)]$ form a $\KK$-basis of $B_\KK^A(G)$ consisting of pairwise orthogonal idempotents, the assertion in (a) follows.

\smallskip
(b) If $e_{(K,\Psi)}^G\in F(G)$ and $f\in\Aut(G)$ then $e_{(f(K),\Psi\circ (f|_K)^*)}^G = \isom_f(e_{(K,\Psi)}^G)\in F(G)$.

\smallskip
(c) Assume that $H$ is a minimal group for $F$ and that $(K,\Psi)\in\calX_F(H)$. Then $e_{(K,\Psi)}^H\in F(H)$ and 
\begin{equation*}
    0\neq e_{(K,\Psi)}^K= e_{(K,\Psi)}^K\res^H_K(e_{(K,\Psi)}^H)\in F(K)\,,
\end{equation*}
by Proposition~\ref{prop s comp res}(b). The minimality of $H$ implies $K=H$. By Proposition~\ref{prop m number}(a), the minimality also implies that $m_{(H,\Psi)}^N=0$ for all $\{1\}\neq N\trianglelefteq H$, since $e_{(H,\Psi)}^H\in F(H)$. Clearly, $E_{(H,\Psi)}\subseteq F$.
\end{proof}

\begin{definition}\label{def minimal pair}
Let $F$ be a subfunctor of $B^A_\KK$ in $\calF_{\KK}^A$. If $H$ is a minimal group for $F$ and $\Psi\in\Hom(H,\KK^\times)$ is such that $(H,\Psi)\in\calX_F(H)$ then we call $(H,\Psi)$ a {\em minimal pair} for $F$. By Proposition~\ref{prop min pairs are BA}(c), each minimal pair for $F$ is a $B^A$-pair.
 \end{definition}

\begin{proposition}\label{prop min pairs of E}
Let $H$ be a finite group, $\Psi\in\Hom(H^*,\KK^*)$, and let $(G,\Phi)$ be a minimal pair for $E_{(H,\Psi)}$. Then:

\smallskip
{\rm (a)} $E_{(H,\Psi)}=E_{(G,\Phi)}$.

\smallskip
{\rm (b)} There exists $N\trianglelefteq H$ with $m_{(H,\Psi)}^N\neq 0$ and $(H/N,\Psi_N)\cong(G,\Phi)$. In particular $(G,\Phi)\preccurlyeq (H,\Psi)$. Moreover, if also $N'\trianglelefteq H$ satisfies $(H/N',\Psi_{N'})\cong (G,\Phi)$ then $m_{(H,\Psi)}^{N'} = m_{(H,\Psi)}^N\neq 0$.

\smallskip
{\rm (c)} Up to isomorphism, $(G,\Phi)$ is the only minimal pair for $E_{(H,\Psi)}$.

\smallskip
{\rm (d)} If $(H,\Psi)$ is a $B^A$-pair, then, up to isomorphism, $(H,\Psi)$ is the only minimal pair of $E_{(H,\Psi)}$. In particular,
\begin{equation*}
   E_{(H,\Psi)}(H)=\bigoplus_{\substack{(H,\Psi')\in\calXhat(H)\\
   (H,\Psi')=_{\Out(H)}(H,\Psi)}} \KK e_{(H,\Psi')}^H\,.
\end{equation*}
\end{proposition}

\begin{proof}
(b) Since $(G,\Phi)$ is a minimal pair for $E_{(H,\Psi)}$, there exists $x\in B_\KK^A(G,H)$ such that $e_{(G,\Phi)}^G= x\cdot_H e_{(H,\Psi)}$. Multiplication with $e_{(G,\Phi)}^G$ yields $e_{(G,\Phi)}^G = e_{(G,\Phi)}^G\cdot(x\cdot_H e_{(H,\Psi)}^H)$. Thus, there exists $(U,\phi)\in\calM(G\times H)$ with 
\begin{equation*}
  e_{(G,\Phi)}^G \cdot \left( \left[\frac{G\times H}{U,\phi}\right]\cdotH e_{(H,\Psi)}^H \right) \neq 0\,.
\end{equation*}
Lemma~\ref{lem double prod} implies that $p_1(U)=G$, $p_2(U)=H$, $\phi$ has an extension to $G\times H$, $\Phi_{k_1(U)}=\Psi_{k_2(U)}\circ\eta_U^*$, and $m_{(H,\Psi)}^{k_2(U)}\neq 0$.
Since $\phi$ has an extension to $G\times H$, $\left[\frac{G\times H}{U,\phi}\right]$ factors through $q(U)\cong G/k_1(U)$ by Lemma~\ref{lem ext and fact}. Since $G$ is a minimal group for $E_{(H,\Psi)}$ this implies $k_1(U)=\{1\}$. Set $N:=k_2(U)\trianglelefteq H$. Then $\eta_U\colon H/N\myiso G$ satisfies $\Psi_N\circ\eta_U^*=\Phi$. If also $N'\trianglelefteq G$ satisfies $(H/N',\Psi_{N'})\cong (G,\Phi)$, then ($H/N,\Psi_{N'})\cong (H/N,\Psi_N)$ and we obtain $m_{(H,\Psi)}^{N'}=m_{(H,\Psi)}^N\neq 0$ by Proposition~\ref{prop G/M isomorphic G/N}.

\smallskip
(a) Note that $E_{(G,\Phi)}\subseteq E_{(H,\Psi)}$ by Proposition~\ref{prop min pairs are BA}(c). The converse follows from Proposition~\ref{prop E in E}(a) and Part~(b).

\smallskip
(c) Assume that also $(G',\Phi')$ is a minimal pair for $E_{(H,\Phi)}=E_{(G,\Phi)}$. Then, by Part~(b) applied to $(G',\Phi')$ and $(G,\Phi)$ in place of $(G,\Phi)$ and $(H,\Psi)$, we have $(G',\Phi')\preccurlyeq (G,\Phi)$. Since both $G$ and $G'$ are minimal groups for $E_{(H,\Psi)}$, they have the same order. Thus, $(G',\Phi')\cong(G,\Phi)$.

\smallskip
(d) Now assume that $(H,\Psi)$ is a $B^A$-pair and let $(G,\Phi)$ be a minimal pair for $E_{(H,\Psi)}$. Then, by Part~(b), there exists $N\trianglelefteq G$ with $m_{(H,\Psi)}^N\neq 0$ and $(H/N,\Psi_N)\cong (G,\Phi)$. Since $(H,\Psi)$ is a $B^A$-pair, this implies $N=\{1\}$ and $(H,\Psi)\cong (G,\Phi)$.
\end{proof}

\begin{notation}\label{not beta}
For any finite group $G$ and any $\Phi\in\Hom(G^*,\KK^\times)$ we denote by $\beta(G,\Phi)$ the class of all minimal pairs for $E_{(G,\Phi)}$. Thus $\beta(G,\Phi)$ is the isomorphism class $[H,\Psi]$ of a $B^A$-pair $(H,\Psi)$. Note that $\beta(G,\Phi)\preccurlyeq [G,\Phi]$ by Proposition~\ref{prop min pairs of E}(b).
\end{notation}

The following proposition is not used in this paper, but of interest in its own right. It is the analogue of \cite[Theorem~5.4.11]{Bouc2010}.

\begin{proposition}
Let $G$ be a finite group and let $\Phi\in\Hom(G^*,\KK^\times)$.

\smallskip
{\rm (a)} If $(H,\Psi)$ is a $B^A$-pair with $(H,\Psi)\preccurlyeq(G,\Phi)$ then $[H,\Psi]\preccurlyeq\beta(G,\Phi)$.

\smallskip
{\rm (b)} For any $N\trianglelefteq G$ the following are equivalent:

\smallskip
\quad {\rm (i)} $m_{(G,\Phi)}^N\neq 0$.

\smallskip
\quad {\rm (ii)} $\beta(G,\Phi)\preccurlyeq [G/N,\Phi_N]$.

\smallskip
\quad {\rm (iii)} $\beta(G,\Phi)=\beta(G/N,\Phi_N)$.

\smallskip
{\rm (c)} For any $N\trianglelefteq G$ the following are equivalent:

\smallskip
\quad {\rm (i)} $[G/N,\Phi_N]\cong \beta(G,\Phi)$.

\smallskip
\quad {\rm (ii)} $(G/N,\Phi_N)$ is a $B^A$-pair and $m_{(G,\Phi)}^N\neq 0$.
\end{proposition}

\begin{proof}
Let $(K,\Theta)\in\beta(G,\Phi)$. Thus, $E_{(G,\Phi)}=E_{(K,\Theta)}$ by Proposition~\ref{prop min pairs of E}(a) and $(K,\Theta)$ is a $B^A$-pair.

\smallskip
(a)  Let $(H,\Psi)$ be as in the statement. Then $E_{(K,\Theta)}=E_{(G,\Phi)}\subseteq E_{(H,\Psi)}$ by Proposition~\ref{prop E in E}(a). Now Proposition~\ref{prop E in E}(b) implies $(H,\Psi)\preccurlyeq(K,\Theta)$.

\smallskip
(b) (i)$\Rightarrow$(ii): Since $m_{(G,\Phi)}^N\neq 0$, Proposition~\ref{prop m number}(a) implies that  
\begin{equation*}
   e_{(G/N,\Phi_N)}^{G/N}= (m_{(G,\Phi)}^N)^{-1}\defl^G_{G/N}(e_{(G,\Phi)}^G)\in E_{(G,\Phi)}(G/N)
   =E_ {(K,\Theta)}(G/N)
\end{equation*} 
so that $E_{(G/N,\Phi_N)}\subseteq E_{(K,\Theta)}$. Proposition~\ref{prop E in E}(b) implies $(K,\Psi)\preccurlyeq (G/N,\Phi_N)$.

\smallskip
(ii)$\Rightarrow$(iii): By Part~(a) applied to $(K,\Theta)$ and $(G/N,\Phi_N)$ we obtain $\beta(G,\Phi)=[K,\Theta]\preccurlyeq\beta(G/N,\Phi_N)$. Conversely, we have $\beta(G/N,\Phi_N)\preccurlyeq [G/N,\Phi_N] \preccurlyeq [G,\Phi]$ and Part~(a) again implies $\beta(G/N,\Phi_N)\preccurlyeq \beta(G,\Phi)$.

\smallskip
(iii)$\Rightarrow$(i): By Proposition~\ref{prop min pairs of E}(b) there exists $M\trianglelefteq G$ such that $m_{(G,\Phi)}^M\neq 0$ and $[G/M,\Phi_M]=\beta(G,\Phi)$. Similarly, there exists $N\le M'\trianglelefteq G$ such that $m_{(G/N,\Phi_N)}^{M'/N}\neq 0$ and $[(G/N)/(M'/N),(\Phi_N)_{M'}]=\beta(G/N,\Phi_N)$. Since
\begin{equation*}
   [G/M',\Phi_{M'}]=[(G/N)/(M'/N),(\Phi_N)_{M'}] = \beta(G/N,\Phi_N) = \beta(G,\Phi) = [G/M,\Phi_M]\,,
\end{equation*}
Proposition~\ref{prop G/M isomorphic G/N} implies that $m_{(G,\Phi)}^{M'} = m_{(G,\Phi)}^M\neq 0$. By Proposition~\ref{prop trans m formula} we have $m_{(G,\Phi)}^{M'} = m_{(G,\Phi)}^N \cdot m_{(G/N,\Phi_N)}^{M'}$ which implies that $m_{(G,\Phi)}^N\neq 0$.

\smallskip
(c) This follows immediately from the equivalence between (i) and (iii) in Part~(b), noting that $\beta(G/N,\Phi_N)=[G/N,\Phi_N]$ if $(G/N,\Phi_N)$ is a $B^A$-pair and that $\beta(G,\Phi)$ consists of $B^A$-pairs.
\end{proof}

\begin{definition}
A subset $\calZ$ of the poset $\calB^A$, ordered by the relation $\preccurlyeq$ (cf. Definition~\ref{def BA-pair}), is called {\em closed} if for every $[H,\Psi]\in\calZ$ and $[G,\Phi]\in\calB^A$ with $[H,\Psi]\preccurlyeq[G,\Phi]$ one has $[G,\Phi]\in\calZ$.
\end{definition}

\begin{theorem}\label{thm lattice iso}
Let $\calS$ denote the set of subfunctors of $B_\KK^A$ in $\calF_\KK^A$, ordered by inclusion of subfunctors, and let $\calT$ denote the set of closed subsets of $\calB^A$, ordered by inclusion of subsets. The map
\begin{equation*}
   \alpha\colon \calS\to\calT\,,\quad F\mapsto \{[H,\Psi]\in\calB^A\mid E_{(H,\Psi)}\subseteq F\}
\end{equation*}
is an isomorphism of posets with inverse given by
\begin{equation*}
   \beta\colon \calT\to\calS\,,\quad \calZ\mapsto \sum_{[H,\Psi]\in\calZ} E_{(H,\Psi)}\,.
\end{equation*}
\end{theorem}

\begin{proof}
Clearly, $\alpha$ and $\beta$ are order-preserving.
Let $F\in\calS$. By Proposition~\ref{prop min pairs are BA}(a) we have
\begin{equation*}
   F= \sum_{\substack{G\\ (H,\Psi)\in\calX_F(G)}} \langle e_{(H,\Psi)}^G\rangle\,,
\end{equation*}
where $G$ runs through a set of representatives of the isomorphism classes of finite groups and $\langle e_{(H,\Psi)}^G\rangle$ denotes the subfunctor of $B_\KK^A$ generated by $e_{(H,\Psi)}^G$. For any finite group $G$ and any $(H,\Psi)\in\calX(G)$ one has $e_{(H,\Psi)}^G\in F(G)$ if and only if $e_{(H,\Psi)}^H\in F(H)$. In fact, $e_{(H,\Psi)}^H = e_{(H,\Psi)}\cdot\res^G_H(e_{(H,\Psi)}^G)$ by Proposition~\ref{prop s comp res} and $e_{(H,\Psi)}^G\in\KK\cdot\ind_H^G(e_{(H,\Psi)}^H)$ by Proposition~\ref{prop ind of e}. Thus
\begin{equation*}
   F= \sum_{\substack{H\\ (H,\Psi)\in\calXhat_F(H)}} E_{(H,\Psi)}\,,
\end{equation*}
where $H$ runs again through a set of representatives of the isomorphism classes of finite groups and $\calXhat_F(H)=\calXhat(H)\cap \calX_F(H)$. By Propositions~\ref{prop min pairs are BA}(c) and \ref{prop min pairs of E}(a), we obtain
\begin{equation*}
   F=\sum_{\substack{[H,\Psi]\in\calB^A\\ (H,\Psi)\in\calX_F(H)}} E_{(H,\Psi)} = \sum_{[H,\Psi]\in\alpha(F)} E_{(H,\Psi)} 
   =\beta(\alpha(F)) \,,
\end{equation*}
since $(H,\Psi)\in\calX_F(H)$ if and only if $E_{(H,\Psi)}\subseteq F$.

\smallskip
Let $\calZ$ be a closed subset of $\calB^A$. By definition of $\alpha$ and $\beta$ we have
\begin{equation*}
   \alpha(\beta(\calZ)) = \{ [H,\Psi]\in\calB^A\ \mid E_{(H,\Psi)}\subseteq \sum_{[G,\Phi]\in\calZ} E_{(G,\Phi)}\}\,.
\end{equation*}
The inclusion $\calZ\subseteq \alpha(\beta(\calZ))$ is obvious. Conversely, assume that $[H,\Psi]\in\calB^A$ satisfies $E_{(H,\Psi)}\subseteq \sum_{[G,\Phi]\in\calZ} E_{(G,\Phi)}$. Evaluation at $H$ and Proposition~\ref{prop min pairs are BA}(a) imply that there exists $[G,\Phi]\in\calZ$ with $e_{(H,\Psi)}^H\in E_{(G,\Phi)}(H)$, which implies $E_{(H,\Psi)}\subseteq E_{(G,\Phi)}$. Since $(G,\Phi)$ is a $B^A$-pair, Proposition~\ref{prop E in E}(b) implies $[G,\Phi]\preccurlyeq [H,\Psi]$. Since $[G,\Phi]\in\calZ$ and $\calZ$ is closed we obtain $[H,\Psi]\in\calZ$. Thus, $\alpha(\beta(\calZ))\subseteq \calZ$, and the proof is complete.
\end{proof}

\begin{Remark}\label{rem max closed subset}
(a) If $(G,\Phi)$ is a $B^A$-pair, then the subfunctor $E_{(G,\Phi)}$ of $B_\KK^A$ corresponds under the bijection in Theorem~\ref{thm lattice iso} to the subset $\calB^A_{\succcurlyeq [G,\Phi]}:=\{[H,\Psi]\in\calB^A\mid [G,\Phi]\preccurlyeq [H,\Psi]\}$. Clearly, $\calB^A_{\succ [G,\Phi]}:=\{[H,\Psi]\in\calB^A\mid [G,\Phi]\prec [H,\Psi]\}$ is the unique maximal closed subset of $\calB^A_{\succcurlyeq [G,\Phi]}$.

\smallskip
(b) For every element $[G,\Phi]\in\calB^A$ there exist only finitely many elements $[H,\Psi]\in\calB^A$ with $[H,\Psi]\preccurlyeq [G,\Phi]$. Therefore, every non-empty subset of $\calB^A$ has a minimal element.
\end{Remark}


\section{Composition factors of $B_\KK^A$}\label{sec composition factors of BA}

We keep the assumptions on $A$ and $\KK$ from Section~\ref{sec BA-pairs and the subfunctors E}. In this section we show that the composition factors of $B_\KK^A$ are parametrized by isomorphism classes of $B^A$-pairs.

\begin{nothing} Recall from \cite[Section~9]{BoltjeCoskun2018} that the simple $A$-fibered biset functors $S$ over $\KK$ are parametrized by isomorphism classes of certain quadruples $(G,K,\kappa,V)$. Here, $G$ is a minimal group for $S$, $(K,\kappa)\in\calM(G)$ is such that the idempotent $f_{(K,\kappa)}\in B_\KK^A(G,G)$ (see \cite[Subsection~4.3]{BoltjeCoskun2018}) does not annihilate $S(G)$, and $V:=S(G)$ is an irreducible $\KK\Gamma_{(G,K,\kappa)}$-module for the finite group $\Gamma_{(G,K,\kappa)}$ defined in \cite[6.1(c)]{BoltjeCoskun2018}. All we need for the analysis in the following theorem is that the idempotent $f_{(K,\kappa)}$ is in the $\KK$-span of standard basis elements $\left[\frac{G\times G}{U,\phi}\right]$ with $(U,\phi)\in\calM(G,G)$ such that $k_2(U)\ge K$, and that in the case $(K,\kappa)=(\{1\},1)$, the group $\Gamma_{(G,\{1\},1)}$ is the set of standard basis elements $\left[\frac{G\times G}{U,\phi}\right]$ of $B_\KK^A(G,G)$ with $p_1(U)=G=p_2(U)$ and $k_1(U)=\{1\}=k_2(U)$. The multiplication is given by $\cdot_G$. Thus, in this case, $\left[\frac{G\times G}{U,\phi}\right]=\tw_\alpha\cdot_G\isom_f$, where $f:=\eta_U\in\Aut(G)$ is defined by $U=\{(\eta_U(g),g)\mid g\in G\}$, and $\alpha\in G^*$ is given by $\alpha(g)=\phi(\eta_U(g),g)$ for $g\in G$. Mapping $\left[\frac{G\times G}{U,\phi}\right]$ to the element $(\alpha,\fbar)\in G^*\rtimes\Out(G)$ defines an isomorphism. Here, $\Out(G)$ acts on $G^*$ via $\lexp{\fbar}{\alpha}:=\alpha\circ f^*$. Moreover, $\Gamma_{(G,\{1\},1)}$ acts on $S(G)$ by $\cdot_G$.
\end{nothing}

\begin{proposition}\label{prop E modulo J}
Let $(G,\Phi)$ be a $B^A$-pair. The subfunctor $E_{(G,\Phi)}$ of $B_\KK^A$ has a unique maximal subfunctor $J_{(G,\Phi)}$, given by
\begin{equation*}
   J_{(G,\Phi)}=\sum_{[H,\Psi]\in\calB^A_{\succ [G,\Psi]}} E_{(H,\Psi)}\,.
\end{equation*}
The simple functor $S_{(G,\Phi)}:=E_{(G,\Phi)}/J_{(G,\Phi)}$ is isomorphic to $S_{(G,\{1\},1,V_\Phi)}$ where $V_\Phi$ is the irreducible $\KK [G^*\rtimes\Out(G)]$-module 
\begin{equation*}
   V_\Phi:=\Ind_{G^*\rtimes\Out(G)_\Phi}^{G^*\rtimes \Out(G)} (\KK_{\Phitilde})\,,
\end{equation*} 
with $\Phitilde\in\Hom(G^*\rtimes\Out(G)_\Phi,\KK^\times)$ defined by $\Phitilde(\phi,\fbar):=\Phi(\phi)$ for $\phi\in G^*$ and $f\in\Aut(G)$.
\end{proposition}

\begin{proof}
By Remark~\ref{rem max closed subset}(a) and Theorem~\ref{thm lattice iso}, $J_{(G,\Phi)}$ is the unique maximal subfunctor of $E_{(G,\Phi)}$.
Thus, the functor $S:=S_{(G,\Phi)}$ is a simple object in $\calF_\KK^A$. Moreover, $G$ is a minimal group for $S$, since 
$G$ is a minimal group for $E_{(G,\Phi)}$ and $E_{(H,\Psi)}(G)=\{0\}$ for all $[H,\Psi]\in\calB^A_{\succ [G,\Phi]}$. 

Let $(U,\phi)\in\calM(G\times G)$ with $k_2(U)\neq \{1\}$. Then $\left[\frac{G\times G}{U,\phi}\right]$ factors through the group $q(U)$ which has smaller order than $G$. Thus, $\left[\frac{G\times G}{U,\phi}\right]\cdot_G e_{(G,\Phi')}^G=0$ for all $(G,\Phi')\in \calXhat(G)$ with $(G',\Phi')=_{\Out(G)} (G,\Phi)$. By Proposition~\ref{prop min pairs of E}(d) this yields $\left[\frac{G\times G}{U,\phi}\right]\cdot_G S(G)=\{0\}$. Thus, $f_{(K,\kappa)}\cdot_G S(G)=0$ for all $(K,\kappa)$ with $|K|>1$.

This implies that $S$ is parametrized by the quadruple $(G,\{1\},1,V)$, with $V=S(G)$ viewed as $\KK\Gamma_{(G,\{1\},1)}$-module. Since $S(G)$ is the $\KK$-span of the idempotents $e_{(G,\Phi')}^G$, with $(G,\Phi')$ running through the $\Out(G)$-orbit of $(G,\Phi)$, and since $\tw_\alpha\cdot_G\isom_{f}\cdot_G e_{(G,\Phi')}^G=\Phi(\alpha)\cdot e_{(G,\Phi'\circ f^*)}^G$ for all $\alpha\in G^*$ and $f\in\Aut(G)$ and $(G,\Phi')$, the $\KK\Gamma_{(G,\{1\},1)}$-module $S(G)$ is monomial. The stabilizer of the one-dimensional subspace $\KK e_{(G,\Phi)}^G$ is equal to $G^*\rtimes \Out(G)_\Phi$ and this group acts on $\KK e_{(G,\Phi)}^G$ via $\Phitilde$. Thus, $S(G)\cong V_\Phi$ as $\KK\Gamma_{(G,\{1\},1)}$-module and the proof is complete.
\end{proof}

\begin{theorem}\label{thm simple subquotients}
Let $F'\subset F\subseteq B^A_\KK$ be subfunctors in $\calF_\KK^A$ such that $F/F'$ is simple. Then there exists a unique $[G,\Phi]\in\calB^A$ such that $E_{(G,\Phi)}\subseteq F$ and $E_{(G,\Phi)}\not\subseteq F'$. Moreover, $E_{(G,\Phi)}+F'=F$, $E_{(G,\Phi)}\cap F'=J_{(G,\Phi)}$, and $F/F'\cong S_{(G,\Phi)}$.
\end{theorem}

\begin{proof}
Since $\alpha(F')$ is a maximal subset of $\alpha(F)$ and both are closed, it follows from Theorem~\ref{thm lattice iso} and Remark~\ref{rem max closed subset} that $\alpha(F)\smallsetminus\alpha(F')=\{[G,\Phi]\}$ for a unique $[G,\Phi]\in\calB^A$. For any $[H,\Psi]\in\calB^A$ one has $E_{(H,\Psi)}\subseteq F$ and $E_{(H,\Psi)}\not \subseteq F'$ if and only if $[H,\Psi]\in\alpha(F)$ but $[H,\Psi]\notin\alpha(F')$. Thus, the first condition is equivalent to $[H,\Psi]=[G,\Phi]$. Further, we have $F'\subset F'+E_{(G,\Phi)}\subseteq F$ which implies $F'+E_{(G,\Phi})=F$, since $F/F'$ is simple. Thus, $0\neq E_{(G,\Phi)}/(E_{(G,\Phi)}\cap F') \cong (E_{(G,\Phi)}+F')/F' =F/F'$, and by Proposition~\ref{prop E modulo J} we obtain $E_{(G,\Phi)}\cap F'=J_{(G,\Phi)}$ so that $F/F'\cong E_{(G,\Phi)}/J_{(G,\Phi)}\cong S_{(G,\Phi)}$.
\end{proof}


\section{The case $A\le\KK^\times$}\label{sec A in K}

In this section we assume that $A$ is a subgroup of the unit group of a field $\KK$ of characteristic $0$. Then the assumptions on $A$ and $\KK$ from the beginnings of Sections~\ref{sec primitive idempotents}--\ref{sec composition factors of BA} are satisfied. This special case has been used for instance in the theory of canonical induction formulas, see \cite{Boltje1998b}. This assumption was also used in \cite{Barker2004} and \cite{CoskunYilmaz2019}. By double duality it allows to consider pairs $(G,gO(G))$ for a normal subgroup $O(G)$ of $G$ instead of pairs $(G,\Phi)$ with $\Phi\in\Hom(G^*,\KK^\times)$. This section makes this translation precise and also translates previously defined features for pairs $(G,\Phi)$ to features for pairs $(G,gO(G))$.

\medskip
For any finite group $G$ we have a homomorphism
\begin{equation*}
   \zeta_G\colon G\to \Hom(G^*,\KK^\times)\,,\quad g\mapsto \varepsilon_g\,,\quad 
   \text{with $\varepsilon_g(\phi):=\phi(g)$,}
\end{equation*}
for $\phi\in G^*$. Note that $\zeta_G$ is functorial in $G$, i.e., if $f\colon G\to H$ is a group homomorphism then $\zeta_H\circ f= \Hom(f^*,\KK^\times)\circ \zeta_G$.
We set
\begin{equation*}
   O^A(G):=O(G):=\ker(\zeta_G)=\bigcap_{\phi\in G^*}\ker(\phi)\,,
\end{equation*}
which is a normal subgroup of $G$ containing the commutator subgroup $[G,G]$ of $G$. Thus, we obtain an injective homomorphism $\zetabar_G\colon G/O(G)\to \Hom(G^*,\KK^\times)$.

\begin{proposition}\label{prop zeta is surjective}
Let $G$ be a finite group.

\smallskip
{\rm (a)} The homomorphism $\zeta_G$ is surjective and $\zetabar_G\colon G/O(G)\myiso \Hom(G^*,\KK^\times)$ is an isomorphism.

\smallskip
{\rm (b)} The subgroup $O(G)$ is the smallest subgroup $[G,G]\le M\le G$ such that $A$ has an element of order $\exp(G/M)$.

\smallskip
{\rm (c)} For any normal subgroup $N$ of $G$ one has $O(G/N)=O(G)N/N$. 
\end{proposition}

\begin{proof}
(a) Applying the functoriality with respect to the natural epimorphism $f\colon G\to G/[G,G]$, and using that $f^*$ is an isomorphism, it suffices to show the statement when $G$ is abelian. Since $\Hom(-^*,\KK^\times)$ preserves direct products of abelian groups, we are reduced to the case that $G$ is cyclic. Using again the functoriality with respect to the natural epimorphism onto the largest quotient of $G$ whose order occurs as an element order in $A$, we are reduced to the case that $G$ is cyclic of order $n$ and $A$ has an element of order $n$. In this case it is easy to see that $\zeta_G$ is injective and that $G$ and $\Hom(G^*,\KK^\times)$ have the same order.

\smallskip
(b) First note that if $M_1$ and $M_2$ have the stated property, then also $M_1\cap M_2$ has this property. In fact, $G/(M_1\cap M_2)$ is isomorphic to a subgroup of $G/M_1\times G/M_2$, whose exponent is equal to the order of an element in $A$. Here we use that if elements $a$ and $b$ in $A$ have orders $k$ and $l$ respectively, then $A$ has an element whose order is the least common multiple of $k$ and $l$. Thus, there exists a smallest subgroup $M$ with the stated property. Clearly, $\ker(\phi)$ has the property for every $\phi\in G^*$. Therefore, also $O(G)$ has the desired property. Conversely, if $M$ has the property, then by writing $G/M$ as a direct product of $n$ cyclic groups whose orders are achieved as element order in $A$, it is easy to construct elements $\phi_1,\ldots,\phi_n\in G^*$ such that $\bigcap_{i=1}^n \ker(\phi_i)= M$, implying that $O(G)\le M$.

\smallskip
(c) Since the exponent of $G/O(G)$ is equal to the order of an element of $A$ also the exponent of $(G/N)/(O(G)N/N)\cong G/(O(G)N)$ is equal to the order of an element of $A$. Thus, $O(G/N)\le O(G)N/N$. Conversely, 
\begin{equation*}
   O(G)=\bigcap_{\phi\in G^*}\ker(\phi) \le \bigcap_{\substack{\phi\in G^*\\ \phi|_N =1}} \ker(\phi)\,,
\end{equation*}
and taking images in $G/N$ yields the reverse inclusion.
\end{proof}

For any finite group $G$, Proposition~\ref{prop zeta is surjective}(a) yields a bijection between the set of pairs of the form $(G,\Phi)$, with $\Phi\in\Hom(G^*,\KK^\times)$ and the set of pairs $(G,gO(G))$ with $gO(G)\in G/O(G)$. More precisely, we identify $(G,gO(G))$ with $(G,\varepsilon_g)$. The following proposition translates various relevant features of pairs $(G,\Phi)$ to features of the corresponding pairs $(G,gO(G))$. The proofs are straightforward and left to the reader.

\begin{proposition}\label{prop new pairs}
Let $G$ and $H$ be finite groups.

\smallskip
{\rm (a)} Let $g\in G$ and $h\in H$. Then $(G,gO(G))\cong (H,hO(H))$ if and only if there exists an isomorphism $f\colon G\to H$ such that $f(g)O(H)= hO(H)$.

\smallskip
{\rm (b)} Let $N$ be a normal subgroup of $G$, let $g\in G$, and set $\Phi:=\varepsilon_g$. Then $\Phi_N=\varepsilon_{gN}$.

\smallskip
{\rm (c)} Let $g\in G$ and $h\in H$. Then $(H,\varepsilon_h)\preccurlyeq(G,\varepsilon_g)$ if and only if there exists a normal subgroup $N$ of $G$ and an isomorphism $f\colon H\myiso G/N$ with $f(h)\in gO(G)N$.

\smallskip
{\rm (d)} Let $K\le G$ and $g\in G$. Then $\varepsilon_g|_{K^\perp}=1$ if and only if $g\in KO(G)$.
\end{proposition}


\end{document}